\documentclass[11pt]{ip-journal}

\usepackage{bm}
\usepackage{amsfonts,amssymb}
\usepackage{dsfont}
\usepackage{amscd}
\usepackage{extarrows}
\usepackage{amsmath}
\usepackage{enumerate}
\usepackage{amscd}
\usepackage[all]{xy}
\usepackage[hyperfootnotes=true]{hyperref}

\usepackage{mathrsfs}

\newtheorem{thm}{Theorem}[section]

\newtheorem{lemma}[thm]{Lemma}

\newtheorem{prop}[thm]{Proposition}
\newtheorem{definition}{Definition}[section]

\newtheorem{rmk}{Remark}[section]

\begin{document}

\title{CR eigenvalue estimate and Kohn-Rossi cohomology}
\author{Zhiwei Wang}
\address{School
of Mathematical Sciences\\Beijing Normal University\\Laboratory of
Mathematics and Complex Systems
\\Ministry of Education\\ Beijing 100875\\ P. R. China\\ }
\email{zhiwei@bnu.edu.cn}
\author{Xiangyu Zhou}
\address{ Institute of Mathematics\\Academy of Mathematics and Systems Sciences\\and Hua Loo-Keng Key
Laboratory of Mathematics\\Chinese Academy of
Sciences\\Beijing  100190\\ P. R. China\\}
\address{School of
Mathematical Sciences\\University of Chinese Academy of Sciences\\
Beijing 100049\\ P. R. China\\}
\email{xyzhou@math.ac.cn}


\begin{abstract}Let $X$ be  a compact connected   CR manifold with a transversal
CR $S^1$-action of real dimension $2n-1$, which is only assumed to be weakly pseudoconvex. Let $\Box_b$ be the $\overline{\partial}_b$-Laplacian, with respect to a $T$-rigid Hermitian metric (see Definition \ref{defn: T rigid} of $T$-rigid Hermitian metric). Eigenvalue estimate of $\Box_b$ is a fundamental issue both in CR geometry and   analysis. In this paper, 
we are able to obtain a sharp estimate of the number of
eigenvalues smaller than or equal to $\lambda$ of $\Box_b$ acting  on the $m$-th Fourier components of smooth
$(n-1,q)$-forms on $X$, where $m\in \mathbb{Z}_+$ and $q=0,1,\cdots,
n-1$. Here the sharp means the  growth order with respect to $m$ is sharp. In particular, when $\lambda=0$, we obtain the asymptotic estimate of the growth for $m$-th Fourier components $H^{n-1,q}_{b,m}(X)$ of $H^{n-1,q}_b(X)$ as $m
\rightarrow +\infty$.  Furthermore,  we establish a Serre type duality theorem for Fourier components of Kohn-Rossi cohomology which is of independent interest.  As a byproduct, the asymptotic growth of  the dimensions of the Fourier components
$H^{0,q}_{b,-m}(X)$  for $
m\in \mathbb{Z}_+$  is established. 
We also give appilcations of our main results, including
Morse type inequalities, asymptotic Riemann-Roch type theorem,
Grauert-Riemenscheider type criterion, and an orbifold version of
our main results which provides an  answer  towards  a  folklore  open problem informed to us by Hsiao.
\end{abstract}
\thanks{	This research is supported by National Key R\&D Program of China (No. 2021YFA1002600 and No. 2021YFA1003100). The authors are partially supported respectively by NSFC grants (11701031, 12071035, 11688101,11431013).  }
\maketitle



\section{Introduction}

Let $(X, T^{1,0}X)$ be a compact connected CR manifold of real dimension $2n-1$, $n\geq 2$, where $T^{1,0}X$
is the given CR structure on $X$.  For $p,q=0,1,\cdots, n-1$,  let $\Omega^{p,q}(X)$ be the space of
smooth $(p,q)$-forms on $X$.  Associated to the CR structure, there naturally comes 
the tangential Cauchy-Riemann operator $\overline{\partial}_b$, which  shares  similar properties as $\overline{\partial}$-operator on complex manifolds.  For example,    $\overline{\partial}_b^2=0$.
It  induces a complex   $\{\Omega^{p,\bullet}(X),\overline{\partial}_b\}$  which is called
$\overline{\partial}_b$-complex, and an  intrinsic  cohomology theory  known  as   Kohn-Rossi cohomology, say  $H^{p,q}_b(X)$  defined as
$H^{p,q}_b(X):=\mbox{Ker}\overline{\partial}_b|_{\Omega^{p,q}(X)}/\mbox{Im}\overline{\partial}_b|_{\Omega^{p,q-1}}(X)$. For references, we refer to \cite{Bo91, CS01, Ja90,KR65}.

One of the most significant feature of differential geometry is the input of a specific metric. CR geometry is not an exception. Putting a Hermitian metric into the CR strcuture, in a quite standard way, one get the so called $\Box_b$-operator.  

The mystry of the geometry of CR mainfold lies behind   $\overline{\partial}_b$  and $\Box_b$-operator.

 One may expect many Riemannian geometric methods, which have shown great power in the study of fundamental problems in Riemannian geometry, such as gradient estimate, eigenvalue estimate,  the analysis of heat kernel  and so on, can be adapted  into the CR picture. 

However, both the   $\overline{\partial}_b$  and $\Box_b$ operators are not elliptic operators, which put much more thorns on the way  for  CR geometers. 

For the time being, there are many progress towards the study of  $\overline{\partial}_b$  and $\Box_b$, or in other words, the study of Kohn-Rossi cohomology and CR eigenvalue estimate, most of which focus on strongly pseudoconvex CR manifolds.  For example,  the embedding problem of the CR manifolds (c.f. \cite{Ak72,
AS70,FK72, Bl94, BdM1, Bu79, Ep92, JT82, Ko86, Ku82a, Ku82b, Ku82c,
Le92, Ni74, Ro65, We89a, We89b,Hs14}, etc), and the  study of  the isolated
singularities of complex hypersurfaces (c.f. \cite{Mi68, Na74,
Na75a, Na75b, Ta74, Ta75}, etc).    It is worth to mention that
Stephen Yau and his coauthors have found many important and
remarkable applications of the Kohn-Rossi cohomology to the theory
of singularities, complex Plateau problem, the embedding problem of
the CR manifolds, and rigidity problems on CR morphisms between
compact
strongly pseudoconvex CR manifolds (\cite{DY12,LY,Yau,Yau11,YauZuo,YZ17,YZ18}, etc).


Though there are important progress, the CR eigenvalue estimate and Kohn-Rossi cohomology of general CR manifold, for example,  CR manifolds which are not necessarily strongly pseudoconvex, are not so well understood. This is the main concern of this paper.

In this paper, we take the task of studying the CR eigenvalue estimate and Kohn-Rossi cohomology of  a
compact connected CR manifold $X$ admitting a
transversal  CR $S^1$-action, which is only assumed to be weakly pseudoconvex.

Thanks to the  transversal CR $S^1$-action,  we have the Fourier decomposition 
  $\Omega^{p,q}(X)=\oplus_{m\in
\mathbb{Z}}\Omega^{p,q}_m(X)$, where $\Omega^{p,q}_m(X)$ is the
$m$-th Fourier component of  $\Omega^{p,q}(X)$ with respect to the
$S^1$-action, and the $\overline{\partial}_b$ operator acts on the
graded algebra $\oplus_q\Omega^{p,q}_m(X)$. One can thus define the
$m$-th Fourier component $H^{p,q}_{b,m}(X)$   of  $(p,q)$-th
Kohn-Rossi cohomology group $H^{p,q}_b(X)$.  Let $\Box^{p,q}_{b,m}$ be the restriction of the
$\overline{\partial}_b$-Laplace operator to the space
$\Omega^{p,q}_m(X)$, which turns out to be a self-adjoint  operator.  Let $\mathcal{H}^{p,q}_{b,m,\leq \lambda}$ be
the linear span of the eigenforms of $\Box^{p,q}_{b,m}$ in
$\Omega^{p,q}_m(X)$ with eigenvalues smaller than or equal to
$\lambda$.  By  a  Hodge type  theory, $\mathcal{H}^{p,q}_{b,m,\leq
	0}(X):=\mathcal{H}^{p,q}_{b,m}(X)$ is the space of
$\Box^{p,q}_{b,m}$ harmonic forms, and 
isomorphic to $H^{p,q}_{b,m}(X)$. In particular,
$H^{p,q}_{b,m}(X)$   is of finite dimension for every $m\in
\mathbb{Z}$. \\

The first main result of this paper is the following asymptotic
estimate for the distribution of eigenvalues of
$\Box^{n-1,q}_{b,m}$.

\begin{thm}\label{main theorem} Let $(X, T^{1,0}X)$ be a compact connected CR manifold of dimension
$2n-1$, $n\geq 2$, where $T^{1,0}X$ is the given CR structure on $X$. Assume that $X$ admits a transversal
CR $S^1$ action and  $X$ is weakly pseudoconvex. Then for $m$ sufficiently large, if $0\leq \lambda\leq m$,
\begin{align*}\dim \mathcal{H}^{n-1,q}_{b,m,\leq \lambda}\leq C(\lambda+1)^qm^{n-1-q},
\end{align*}
and if $1\leq m\leq \lambda$,
\begin{align*}
\dim \mathcal{H}^{n-1,q}_{b,m,\leq \lambda}\leq C\lambda^{n-1}.
\end{align*}

\end{thm}


It is worth to mention that, examples are provided in Section \ref{proof}
to show  that the growth order, say $m^{n-1-q}$ in our  Theorem \ref{main theorem}
can not be improved in general.

Taking account of the  transversal CR $S^1$-action,  we want to ask more structures for the Kohn-Rossi cohomology. To this end, we get our second main result as follows.

\begin{thm}[Serre type duality theorem]\label{Serre duality I}Let $X$ be a compact $CR$-manifold of real
dimension $2n-1$, which admits a transversal CR $S^1$-action which is locally free. Then  we have the
following conjugate linear isomorphism in the cohomologcial level
\begin{align*}
H^{p,q}_{b,m}(X)\simeq H^{n-1-p,n-1-q}_{b,-m}(X).
\end{align*}
\end{thm}

With Serre type duality in hand, Theorem \ref{main theorem} gives us  the following
\begin{thm}\label{improved estimate}
Let $X$ be a compact $CR$-manifold of real dimension $2n-1$, which admits a transversal $S^1$-action which
is locally free. Suppose that $X$ is a weakly pseudoconvex CR manifold. Then we have that 
\begin{align*}
\dim H^{0,q}_{b,-m}(X)\leq Cm^{q}, ~~~~\mbox{~~~as~~~}~~~~m\rightarrow +\infty.
\end{align*}
\end{thm}

It is worth to mention that Theorem \ref{improved estimate} improves   the corresponding  estimates of Hsiao and Li in \cite{HL16}. Namely, under the same assumption of Theorem \ref{improved estimate}, Hsiao and Li obtained the following estimate:
$\dim H^{0,q}_{b,-m}(X)=o(m^{n-1})$ as $m\rightarrow +\infty$ for $0\leq q<n-1$.\\

Several applications of our main result are in order. 
Firstly, combining Morse type inequalities in \cite{HL16}
with Theorem \ref{Serre duality I}, there comes the following
\begin{thm}\label{weak m negative-1}Let $X$ be a compact connected CR manifold with a transversal CR $S^1$-action.
Assume that dim$_{\mathbb{R}}X=2n-1$, $n\geq 2$. Then for every $q=0,1,\cdots,n-1$, as $m\rightarrow +\infty$,
we have
\begin{align*}
	&\mbox{dim}H^{n-1, q}_{b,-m}(X)\leq \frac{m^{n-1}}{2\pi^n}\int_{X(n-1-q)}|\det \mathcal{L}_x|dv_{X}(x)+o(m^{n-1}), \\
	&\sum_{j=n-1-q}^{n-1}(-1)^{j+q-(n-1)} \mbox{dim}H^{n-1,j}_{b,-m}(X)\\
	&\leq \frac{m^{n-1}}{2\pi^n}\sum^{n-1}_{j=n-1-q}(-1)^{q+j-(n-1)}\int_{X(n-1-j)}|\det\mathcal{L}_x|dv_X(x)+o(m^{n-1}).
\end{align*}
In particular, when $q=n-1$, as $m\rightarrow+ \infty$, we have the asymptotic Riemann-Roch theorem
\begin{align*}
	&\sum^{n-1}_{j=0}(-1)^{n-1-j}\dim H^{n-1,j}_{b,m}(X)\\
	&=\frac{m^{n-1}}{2\pi^n}\sum_{j=0}^{n-1}(-1)^{n-1-j}\int_{X(n-1-j)}|\det \mathcal{L}_x|dv_X(x)+o(m^{n-1}).
\end{align*}
\end{thm}

\begin{thm} \label{morse m positive-1}Let $X$ be a compact connected CR manifold with a transversal CR
$S^1$-action. Assume that $\dim_{\mathbb{R}}X=2n-1$, $n\geq 2$. For every $q=0,1,2,\cdots, n-1$, as
$m\rightarrow +\infty$, we have
\begin{align*}
	&\dim H^{n-1,q}_{b,m}(X)\leq \frac{m^{n-1}}{2\pi^n}\int_{X(q)}|\det\mathcal{L}_x|dv_X(x)+o(m^{n-1}),\\
	&\sum^{n-1}_{j=n-1-q}(-1)^{q+j-(n-1)}\dim H^{n-1,j}_{b,m}(X)\\
	&\leq \frac{m^{n-1}}{2\pi^n}\sum^{n-1}_{j=n-1-q}(-1)^{q+j-(n-1)}\int_{X(j)}|\det\mathcal{L}_x|dv_X(x)+o(m^{n-1}).
\end{align*}
In particular, when $q=n-1$, as $m\rightarrow +\infty$, we have the following asymptotic Riemann-Roch theorem
\begin{align*}
	&\sum^{n-1}_{j=0}(-1)^{n-1-j}\dim H^{n-1,j}_{b,m}(X)\\
	&=\frac{m^{n-1}}{2\pi^n}\sum_{j=0}^{n-1}(-1)^{n-1-j}\int_{X(j)}|\det \mathcal{L}_x|dv_X(x)+o(m^{n-1}).
\end{align*}
\end{thm}

\begin{rmk}As a complement of Theorem \ref{weak m negative-1} and Theorem \ref{morse m positive-1}, we point out that,  by almost the same proof as in \cite{HL16}, under the same assumptions, one can prove the corresponding strong Morse type
inequalities for the complex $(\Omega^{n-1,\bullet}_{b,m}(X), \overline{\partial}_{b})$ as follows.
For every $q=0,1,\cdots,n-1$, as $m\rightarrow +\infty$, we have
\begin{align*}
\sum_{j=0}^q(-1)^{q-j} &\mbox{dim}H^{n-1,j}_{b,m}(X)\\
&\leq \frac{m^{n-1}}{2\pi^n}\sum^{q}_{j=0}(-1)^{q-j}\int_{X(j)}|\det\mathcal{L}_x|dv_X(x)+o(m^{n-1});
\end{align*}
and as $m\rightarrow -\infty$, we have
\begin{align*}
&\sum^q_{j=0}(-1)^{q-j}\dim H^{n-1,j}_{b,m}(X)\\
&\leq \frac{|m|^{n-1}}{2\pi^n}\sum^{q}_{j=0}(-1)^{q-j}\int_{X(n-1-j)}|\det\mathcal{L}_x|dv_X(x)+o(|m|^{n-1}).
\end{align*}
\end{rmk}

Combining Theorem \ref{main theorem} and Theorem \ref{morse m positive-1}, we can obtain the following Grauert-Riemenschneider type criterion

\begin{thm}\label{GR criterion-1}
Let $(X, T^{1,0}X)$ be a compact connected CR manifold of dimension $2n-1$, $n\geq 2$, where $T^{1,0}X$ is the
given CR structure on $X$. Assume that $X$ admits a transversal CR $S^1$-action. If  $X$ is weakly pseudoconvex
and strongly pseudoconvex at a point, then
\begin{align*}
\dim H^{n-1,0}_{b,m}(X)\approx m^{n-1}\mbox{~as~~} m\rightarrow +\infty.
\end{align*}
That is to say, there are a lot of CR sections of the canonical bundle $K_X$ of $X$.
\end{thm}
\begin{rmk}
	Theorem \ref{GR criterion-1} can be used to study the embedding problem for weakly pseudoconvex CR manifolds with transversal CR $S^1$-action, which will be discussed in a future work.
\end{rmk}

Orbifold appears frequently when you do some quotients in algebraic geometry and reductions in mathematical physics, for example in the process of symplectic reduction.  It is also a simplest case of singular space. For compact connected weakly pseudoconvex CR orbifolds with transversal CR $S^1$-actions, we establish  the following

\begin{thm}\label{main theorem orbifold-1} Let $(X, T^{1,0}X)$ be a compact connected CR orbifold of
dimension $2n-1$, $n\geq 2$, where $T^{1,0}X$ is the given CR structure on $X$. Assume that $X$ admits a
transversal CR $S^1$-action and  $X$ is weakly pseudoconvex. Then for $m$ sufficiently large, if $0\leq \lambda\leq m$,
\begin{align*}\dim \mathcal{H}^{n-1,q}_{b,m,\leq \lambda}\leq C(\lambda+1)^qm^{n-1-q},
\end{align*}
and if $1\leq m\leq \lambda$,
\begin{align*}
\dim \mathcal{H}^{n-1,q}_{b,m,\leq \lambda}\leq C\lambda^{n-1}.
\end{align*}

\end{thm}


Moreover, we  establish the following isomorphism of Grauert type in the orbifold level, which we think is of
independent interest.

\begin{thm} \label{identification-1}Let $M$ be a compact complex manifold and $G$ a compact Lie group.
Suppose that $G$ acts on $M$ analytically, locally free and dim$_{\mathbb{C}}M/G=n$. Let $L$ be a $G$-invariant
holomorphic Hermitian line bundle over $M$. Suppose that $L$ admits a locally free $G$-action compatible
with $M$. Take any orbifold Hermitian metric $h^L$ (i.e. $G$-invariant Hermitian metric) on $L$, it induces
an orbifold  Hermitian metric $h^{L^*}$ on $L^*$, set $\widetilde{X}=\{v\in L^*| |v|^2_{h^{L^*}}=1\}$
and $X=\widetilde{X}/G$.
Then for every $p,q=0,1,\cdots,n$ and every $m\in \mathbb{Z}$, there is a bijective map
$A^{(p,q)}_m:\Omega^{(p,q)}_m(X)\rightarrow \Omega^{(p,q)}(M/G, L^m/G)$ such that
$A^{(p,q+1)}_m\overline{\partial}_{b,m}=\overline{\partial}A^{(p,q)}_m$ on $\Omega^{(p,q)}_m(X)$. Thus we have that
\begin{align*}
\Omega^{p,q}_m(X)&\simeq \Omega^{p,q}(M/G,L^m/G)\\
H^{p,q}_{b,m}(X)&\simeq H^{p,q}(M/G,L^m/G).
\end{align*}
In particular, $\dim H^{p,q}_{b,m}(X)<\infty$.
\end{thm}

It should be remarked that the case of $p=0$ of Theorem \ref{identification-1} was established in \cite{CHT15}.
From Theorem \ref{main theorem orbifold-1} and Theorem \ref{identification-1}, we obtain the following

\begin{thm}\label{orbifold cohomology-1}
Let $M$ be a compact complex manifold and $G$ a compact Lie group. Suppose that $G$ acts on $M$ analytically,
locally free and dim$_{\mathbb{C}}M/G=n$. Let $(L,h^L)$ be a $G$-invariant holomorphic Hermitian line bundle over $M$.
Suppose that $L$ admits a locally free $G$-action compatible with $M$ and the curvature of $L$ is semi-positive.
Then  we have that for $m$ sufficiently large,
\begin{align*}
\dim H^{n,q}(M/G, L^m/G)\leq Cm^{n-q},
\end{align*}
where $C$ is a constant independent of $m$.
\end{thm}

The above theorem corresponds to   Berndtsson's estimate in the orbifold case, which answers a folklore open question, say generalizing Berndtsson's estimate to the orbifold setting,  informed to us by Hsiao.


For reader's convenient, we sketch the proof of Theorem \ref{main theorem}. By  Baouendi-Rothschild-Treves \cite{BRT85}, we get a picture of the local structure of CR manifolds with transversal CR $S^1$ action. In fact,   locally, it is a part of the a circle bundle, i.e. it can be decomposed to  complex ball in $\mathbb C^{n-1}$ times a small angle, say  $U=:B_\varepsilon \times(-\delta,\delta) $. Our first key observation is that, on $U$, $\overline{\partial}_b$-operator, $\Box_b$-operator, and   any  $\overline{\partial}_b$-closed $(n-1,q)$ form admit very good representations, say by ignoring  some rotations, they coincide with the $\overline{\partial}$-operator, $\Box$-operator and   $\overline{\partial}$-closed $(n-1,q)$-forms on the complex ball $B_\varepsilon$. Our second key observation is  that the weakly pseudoconvexity condition provides us a local potential, which is plurisubharmonic on $B_\varepsilon$.  All these observations inspire  us to  construct trivial holomorphic line bundle over $B_\varepsilon$, and let the local potential serve as a Hermitian metric of this line bundle, and then we can translate  the CR eigenvalue estimate problem to a counterpart in the category of complex holomorphy. To do this,   we modify Berndtsson's trick, do careful analysis.  Finally, Theorem \ref{main theorem} turns up in this paper.

%
The structure of this paper is organized  as follows. In Section
\ref{preliminaries}, we introduce the basics of CR manifolds with transversal CR $S^1$-action. In Section \ref{sect: Hermitian CR}, we introduce  Hermitian CR geometry under transversal CR $S^1$-action. In Section  \ref{sect: local picture}, we formulate the local picture of compact connected weakly pseudoconvex CR manifold with transversal CR $S^1$-action by using Baouendi-Rothschild-Treves' theory. In Section \ref{sect: local rep}, we give a local representation of $\overline{\partial}_b$, $\overline{\partial}^*_b$ and $\Box^{(p,q)}_{b,m}$.  In Section \ref{scaling}, we prepare the scaling technique for the proof of Theorem \ref{main theorem}. In Section \ref{proof}, we give the proof of the
Theorem \ref{main theorem}, and explain why the estimate can not be
improved in general.  In Section \ref{Serre}, we give a proof of a
Serre type duality theorem, namely Theorem \ref{Serre duality I}. In Section \ref{application}, we give some
applications of our main results. Namely, we prove Theorem \ref{weak
m negative-1}, Theorem \ref{morse m positive-1}, Theorem \ref{GR
criterion-1}, Theorem \ref{main theorem orbifold-1}, Theorem
\ref{identification-1}, Theorem \ref{orbifold cohomology-1}.


\section{CR manifold with transversal CR $S^1$-action}\label{preliminaries}
Let $(X, T^{1,0}X)$ be a compact connected CR manifold of dimension $2n-1$, $n\geq 2$, where $T^{1,0}X$
is the given CR structure on $X$. That is, $T^{1,0}X$ is a sub-bundle of the complexified tangent bundle
$\mathbb{C}TX$ of rank $n-1$, satisfying $T^{1,0}X\cap T^{0,1}X=\{0\}$, where $T^{0,1}X=\overline{T^{1,0}X}$,
and $[\mathcal{V},\mathcal{V}]\subset \mathcal{V}$, where $\mathcal{V}=\mathcal{C}^\infty(X,T^{1,0}X)$.

We assume throughout this paper that, $(X,T^{1,0}X)$ is a compact connected CR manifold with a transversal
CR  $S^1$-action.

Denote by $e^{i\theta}~(0\leq \theta<2\pi)$ the $S^1$-action: $S^1\times X\rightarrow X$, $(e^{i\theta},x)\mapsto e^{i\theta}\circ x$. Set $X_{reg}=\{x\in X:\forall  e^{i\theta}\in S^1$, if $e^{i\theta}\circ x=x$, then $e^{i\theta}=$ id$\}$.
We call $x\in X_{reg}$ a regular point of the $S^1$-action.  It is proved in \cite{HL16} that
$X_{reg}$ is an open, dense subset of $X$, and thus the measure of $X\setminus X_{reg}$ is zero.

Let $T\in \mathcal{C}^\infty(X,TX)$  be the global real vector field induced by the $S^1$-action
$e^{i\theta}$ ($\theta\in [0,2\pi)$) given as follows
\begin{align*}
(Tu)(x)=\frac{\partial}{\partial\theta}(u(e^{i\theta}\circ x))\big|_{\theta=0}, u\in \mathcal{C}^\infty(X).
\end{align*}

\begin{definition}
We say that the $S^1$-action is CR if
\begin{align*}
[T,\mathcal{C}^\infty(X,T^{1,0}X)]\subset \mathcal{C}^\infty(X,T^{1,0}X),
\end{align*}
where $[\cdot,\cdot]$ is the Lie bracket between the smooth vector fields on $X$. Furthermore, we say
that the $S^1$-action is transversal if for each $x\in X$,
\begin{align*}
\mathbb{C}T(x)\oplus T^{1,0}_xX\oplus T^{0,1}_xX=\mathbb{C}T_xX.
\end{align*}
\end{definition}

Denote by $\omega_0$ the global real $1$-form determined by $\langle \omega_0,u\rangle=0$,
for every $u\in T^{1,0}X\oplus T^{0,1}X$ and $\langle \omega_0,T\rangle=-1$.

\begin{definition}\label{levi form}
For $x\in X$, the Levi form $\mathcal{L}_x$ associated with the CR structure is the Hermitian quadratic form on
$T^{1,0}_xX$  defined as follows. For any $U,V\in T^{1,0}_xX$, pick
$\mathcal{U},\mathcal{V}\in \mathcal{C}^\infty(X, T^{1,0}X)$ such that $\mathcal{U}(x)=U$, $\mathcal{V}(x)=V$. Set
\begin{align*}
\mathcal{L}_x(U,\overline{V})=\frac{1}{2i}\langle[\mathcal{U},\overline{\mathcal{V}}](x),\omega_0(x)\rangle
\end{align*}
where $[\cdot,\cdot]$ denotes the Lie bracket between smooth vector fields. Note that $\mathcal{L}_x$ does not depend on
the choice of $\mathcal{U}$ and $\mathcal{V}$.
\end{definition}
\begin{definition}
The CR structure on $X$ is called (weakly) pseudo-convex at $x\in X$ if $\mathcal{L}_x$ is positive semi-definite.
It is called strongly pseudoconvex at $x$ if $\mathcal{L}_x$ is positive definite. If the CR structure is
(strongly)  pseudoconvex at every point of $X$, then $X$ is called a (strongly) pseudoconvex CR manifold.
\end{definition}

Denote by $T^{*1,0}X$ and $T^{*0,1}X$ the dual bundle of $T^{1,0}X$ and $T^{0,1}X$ respectively. Define the
vector bundle of $(p,q)$-forms by $T^{*p,q}X:=\Lambda^pT^{*1,0}X\otimes\Lambda^qT^{*0,1}X$. Let $D\subset X$ be
an open subset. Let $\Omega^{p,q}(D)$ denote the space of smooth sections of $T^{*p,q}X$ over $D$ and let
$\Omega^{p,q}_0(D)$ be the subspace of $\Omega^{p,q}(D)$ whose elements have compact support in $D$.

Fix $\theta_0\in [0,2\pi)$. Let
\begin{align*}
de^{i\theta_0}:\mathbb{C}T_xX\rightarrow \mathbb{C}T_{e^{i\theta_0}x}X
\end{align*}
denote the differential map of $e^{i\theta_0}:X\rightarrow X$. By the property of transversal CR $S^1$-action,
one can check that
\begin{align}
&de^{i\theta_0}:T^{1,0}_xX\rightarrow T^{1,0}_{e^{i\theta_0}x}X,\notag\\
&de^{i\theta_0}:T^{0,1}_xX\rightarrow T^{0,1}_{e^{i\theta_0}x}X,\label{tangent map}\\
&de^{i\theta_0}(T(x))=T(e^{i\theta_0}\circ x).\notag
\end{align}
Let $(de^{i\theta_0})^*:\Lambda^{p+q}(\mathbb{C}T^*X)\rightarrow \Lambda^{p+q}(\mathbb{C}T^*X)$ be the pull-back
of $de^{i\theta_0}$, $p, q=0,1,\cdots,n-1$. From (7), we can check that for every $p, q=0,1,\cdots, n-1$,
\begin{align}\label{induced tangent map}
(de^{i\theta_0})^*:T^{*p,q}_{e^{i\theta_0}\circ x}X\rightarrow T^{*p,q}_xX.
\end{align}
Let $u\in \Omega^{p,q}(X)$, define $Tu$ as follows. For any $X_1,\cdots, X_p\in T^{1,0}_xX$ and
$Y_1,\cdots, Y_q\in T_x^{0,1}X$,
\begin{align*}
&Tu(X_1,\cdots,X_p;Y_1,\cdots, Y_q)\\
&:=\frac{\partial}{\partial\theta}((de^{i\theta})^*u(X_1,\cdots,X_p;Y_1,\cdots,Y_q))
|_{\theta=0}.
\end{align*}
From (\ref{tangent map}) and (\ref{induced tangent map}), we have that $Tu\in \Omega^{p,q}(X)$ for all
$u\in \Omega^{p,q}(X)$.

Let $\overline{\partial}_b:\Omega^{p,q}(X)\rightarrow \Omega^{p,q+1}(X)$ be the tangential
Cauchy-Riemann operator. For the definition of tangential Cauchy-Riemann operator, we refer to
\cite{Bo91,CS01, Ja90}. It is straightforward  from (\ref{tangent map}) and (\ref{induced tangent map}) to see that
\begin{align*}
T\overline{\partial}_b=\overline{\partial}_b T
\end{align*}
on $\Omega^{p,q}(X)$. For every $m\in \mathbb{Z}$, put $\Omega^{p,q}_m(X):=\{u\in \Omega^{p,q}(X):Tu=imu\}$.
We have the $\overline{\partial}_b$-complex for every $m\in \mathbb{Z}$:
\begin{align*}
\cdots\rightarrow \Omega^{p,q-1}_m(X)\rightarrow \Omega^{p,q}_m(X)\rightarrow\Omega^{p,q+1}_m(X)\rightarrow \cdots
\end{align*}
For every $m\in \mathbb{Z}$, the  $(p,q)$-th $\overline{\partial}_b$ cohomology (or Kohn-Rossi cohomology) is given by
\begin{align*}
H^{p,q}_{b,m}(X):=\frac{\mbox{Ker}\overline{\partial}_b:\Omega^{p,q}_m(X)\rightarrow \Omega^{p,q+1}_m(X) }
{\mbox{Im}\overline{\partial}_b:\Omega^{p,q-1}_m(X)\rightarrow \Omega^{p,q}_m(X)}.
\end{align*}


For $m\in \mathbb{Z}$, when $q=0$, $H^{0,0}_{b,m}(X)$ is the space of CR functions which lie in the eigenspace
of $T$ and we call $H^{0,0}_{b,m}(X)$ the $m$-th Fourier component of CR functions. We say that a function
$u\in \mathcal{C}^\infty(X)$ is a Cauchy-Riemann (CR for short) function if $\overline{\partial}_bu=0$ or
in the other word, $\overline{Z}u=0$ for all $Z\in \mathcal{C}^\infty(X,T^{1,0}X)$.


\section{Metric input: Hermtian CR geometry under transversal CR $S^1$-action}\label{sect: Hermitian CR}
In this section, we  collect  facts we need  on Hermtian CR geometry under transversal CR $S^1$-action.  Lemmas and Theorems not specified are taken  from \cite{HL15,HL16}.
\begin{definition}
Let $D$ be an open set and let $V\in \mathcal{C}^\infty(D,\mathbb{C}TX)$ be a vector field on $D$. We say
that $V$ is $T$-rigid if
\begin{align*}
de^{i\theta}(V(x))=V(e^{i\theta}\circ x)
\end{align*}
for  any $x,\theta\in [0,2\pi)$ satisfying $x\in D$, $e^{i\theta}\circ x\in D$.
\end{definition}

\begin{definition}\label{defn: T rigid}
Let $\langle\cdot|\cdot\rangle$ be a Hermitian metric on $\mathbb{C}TX$. We say that $\langle\cdot|\cdot\rangle$
is $T$-rigid if for $T$-rigid vector fields $V,W$ on $D$, where $D$ is any open set, we have
\begin{align*}
\langle V(x)|W(x)\rangle=\langle (de^{i\theta}V)(e^{i\theta}\circ x)|(de^{i\theta}W)(e^{i\theta}\circ x)\rangle,
\end{align*}
for any $x\in D$, $ \theta\in [0,2\pi)$ such that $e^{i\theta}\circ x\in D$.
\end{definition}
\begin{lemma}\label{Hermitian metric}Let $X$ be a compact connected CR manifold with a
transversal $S^1$-action. There is always a $T$-rigid Hermitian metric $\langle\cdot|\cdot\rangle$ on
$\mathbb{C}TX$ such that $T^{1,0}X\perp T^{0,1}X$, $T\perp (T^{1,0}X\oplus T^{0,1}X)$, $\langle T|T\rangle=1$
and $\langle u|v\rangle$ is real if $u,v$ are real tangent vectors.
\end{lemma}

From now on, we fix a $T$-rigid Hermitian metric $\langle\cdot|\cdot\rangle$ on $\mathbb{C}TX$ satisfying all
the properties in Lemma \ref{Hermitian metric}. The Hermitian metric $\langle\cdot|\cdot\rangle$ on
$\mathbb{C}TX$ induces by duality a  Hermitian metric on $\mathbb{C}T^*X$ and also on the bundles of
$(p,q)$-forms for $p, q=0,1,\cdots,n-1$. We shall also denote all these induced metrics by
$\langle\cdot|\cdot\rangle$. For every $v\in T^{*p,q}X$, we write $|v|^2:=\langle v|v\rangle$. We have the
pointwise orthogonal decompositions:
\begin{align*}
&\mathbb{C}T^*X=T^{*1,0}X\oplus T^{*0,1}X\oplus \{\lambda\omega_0:\lambda\in \mathbb{C}\},\\
&\mathbb{C}TX=T^{1,0}X\oplus T^{0,1}X\oplus \{\lambda T:\lambda\in \mathbb{C}\}.
\end{align*}

For any $p\in X$, locally there is an orthonormal frame $\{U_1,\cdots, U_{n-1}\}$ of $T^{1,0}X$ with respect
to the given $T$-rigid Hermitian metric $\langle\cdot|\cdot\rangle$  such that the Levi-form $\mathcal{L}_p$
is diagonal in this frame,  $\mathcal{L}_p(U_i,\overline{U}_j)=\lambda_j\delta_{ij}$, where $\delta_{ij}=1$
if $i=j$, $\delta_{ij}=0$ if $i\neq j$. The entries $\{\lambda_1,\cdots,\lambda_{n-1}\}$ are called the
eigenvalues of the Levi-form at $p$ with respect to the $T$-rigid Hermitian metric $\langle\cdot|\cdot\rangle$.
Moreover, the determinant of $\mathcal{L}_p$ is defined by $\det\mathcal{L}_p=\lambda_1(p)\cdots\lambda_{n-1}(p)$.

Let $(\cdot|\cdot)$ be the $L^2$ inner product on $\Omega^{p,q}(X)$ induced by $\langle\cdot|\cdot\rangle$  and
let $\|\cdot\|$ denote the corresponding norm. Then for all  $u,v\in \Omega^{p,q}(X)$
\begin{align*}
(u|v)=\int_X\langle u|v\rangle dv_X
\end{align*}
where $dv_X$ is the volume form on $X$ induced by the $T$-rigid Hermitian metric. Let $L^2_{(p,q),m}(X)$ be
the completion of $\Omega^{p,q}_m(X)$ with respect to $(\cdot|\cdot)$. For $m\in \mathbb{Z}$, let
\begin{align*}
Q^{(p,q)}_m:L^2_{(p,q)}(X)\rightarrow L^2_{(p,q),m}(X)
\end{align*}
be the orthogonal projection with respect to $(\cdot|\cdot)$. Then for any $u\in \Omega^{p,q}(X)$,
\begin{align*}
Q^{(p,q)}_mu=\frac{1}{2\pi}\int_{-\pi}^\pi u(e^{i\theta}\circ x)e^{-im\theta}d\theta.
\end{align*}
By using the elementary Fourier analysis, it is straightforward to see that for any $u\in \Omega^{p,q}(X)$,
\begin{align*}
\sum^N_{m=-N}Q^{(p,q)}_mu\rightarrow u
\end{align*}
in $\mathcal{C}^\infty$ topology as $N\rightarrow \infty$. For every $u\in L^2_{(p,q)}(X)$,
\begin{align*}
\sum^N_{m=-N}Q^{(p,q)}_mu\rightarrow u
\end{align*}
in $L^2_{(p,q)}(X)$ as $N\rightarrow \infty$.
If we denote the $\lim_{N\rightarrow \infty}\sum^N_{m=-N}Q^{(p,q)}_mu$ by $\sum_{m\in \mathbb{Z}}Q^{(p,q)}_mu$,
then we write $u=\sum_{m\in \mathbb{Z}}Q^{(p,q)}_mu$. Thus we have the following Fourier decomposition
\begin{align*}
\Omega^{p,q}(X)=\oplus_{m\in \mathbb{Z}}\Omega^{p,q}_m(X), L^2_{(p,q)}(X)=\oplus_{m\in \mathbb{Z}}L^2_{(p,q),m}(X).
\end{align*}
We have the following Fourier decomposition of the $(p,q)$-th Kohn-Rossi cohomology
\begin{align*}
H^{p,q}_b(X)\simeq \oplus_{m\in \mathbb{Z}}H^{p,q}_{b,m}(X).
\end{align*}

Let $\overline{\partial}^*_b:\Omega^{p,q+1}(X)\rightarrow \Omega^{p,q}(X)$ be the formal adjoint of
$\overline{\partial}_b$ with respect to $(\cdot|\cdot)$. Since the Hermitian metrics $\langle\cdot|\cdot\rangle$
are $T$-rigid, we can check that
\begin{align*}
T\overline{\partial}^*_b=\overline{\partial}^*_bT
\end{align*}
on $\Omega^{p,q}(X)$ for $p, q=0,1,\cdots,n-1$ and thus
\begin{align*}
\overline{\partial}^*_b:\Omega^{p,q+1}_m(X)\rightarrow \Omega^{p,q}_m(X), \forall m\in \mathbb{Z}.
\end{align*}
Put
\begin{align*}
\Box^{(p,q)}_b:=\overline{\partial}_b\overline{\partial}^*_b+\overline{\partial}^*_b\overline{\partial}_b:
\Omega^{p,q}(X)\rightarrow
\Omega^{p,q}(X).
\end{align*}
We also have that
\begin{align*}
T\Box^{(p,q)}_b=\Box^{(p,q)}_bT,
\end{align*}
thus
\begin{align*}
\Box^{(p,q)}_b:\Omega^{p,q}_m(X)\rightarrow \Omega^{p,q}_m(X)
\end{align*}
We will write $\Box^{(p,q)}_{b,m}$ to denote the restriction of $\Box^{(p,q)}_b$ on $\Omega^{p,q}_m(X)$.
For every $m\in \mathbb{Z}$, we extend $\Box^{(p,q)}_{b,m}$ to $L^2_{(p,q),m}(X)$ by
\begin{align*}
\Box^{(p,q)}_{b,m}:Dom(\Box^{(p,q)}_{b,m})\subset L^2_{(p,q),m}(X)\rightarrow L^2_{(p,q),m}(X),
\end{align*}
where $Dom(\Box^{(p,q)}_{b,m})=\{u\in L^2_{(p,q),m}(X):\Box^{(p,q)}_{b,m}u\in L^2_{(p,q),m}(X)$ in the sense
of distribution $\}$.
\begin{thm}\label{regular 1}
For every $s\in \mathbb{N}_0:=\mathbb{N}\cup \{0\}$, there exists a constant $C_s$ such that
\begin{align*}
\|u\|_{s+1}\leq C_s\big(\|\Box^{(p,q)}_bu\|_s+\|Tu\|_s+\|u\|_s\big), \forall u\in \Omega^{p,q}(X),
\end{align*}
where $\|\cdot\|_s$ denotes the standard sobolev norm of order $s$ on $X$.
\end{thm}
\begin{thm}\label{regular 2}
Fix $m\in \mathbb{Z}$, for every $s\in \mathbb{N}_0$, there is a constant $C_{s,m}$ such that
\begin{align*}
\|u\|_{s+1}\leq C_{s,m}\big(\|\Box^{(p,q)}_{b,m}u\|_s+\|u\|_s\big), \forall u\in \Omega^{p,q}_m(X).
\end{align*}
\end{thm}

\begin{thm}\label{hodge}
Fix $m\in \mathbb{Z}$, $\Box^{(p,q)}_{b,m}:Dom(\Box^{(p,q)}_{b,m})\subset L^2_{(p,q),m}(X)\rightarrow L^2_{(p,q),m}(X),$
is  a self-adjoint operator.  The spectrum of $\Box^{(p,q)}_{b,m}$ denoted by Spec$(\Box^{(p,q)}_{b,m})$
is a discrete subset of $[0,+\infty)$. For every $\lambda\in $Spec$(\Box^{(p,q)}_{b,m})$ the eigenspace
with respect to $\lambda$
\begin{align*}
\mathcal{H}^{p,q}_{b,m,\lambda}(X)=\Big\{ u\in Dom(\Box^{(p,q)}_{b,m}): \Box^{(p,q)}_{b,m}u=\lambda u \Big\}
\end{align*}
is finite dimensional with $\mathcal{H}^{p,q}_{b,m,\lambda}(X)\subset \Omega^{p,q}_m(X)$ and for $\lambda=0$
we denote by $\mathcal{H}^{p,q}_{b,m}(X)$ the harmonic space $\mathcal{H}^{p,q}_{b,m,0}(X)$ for brevity and then
we have the Dolbeault isomorphism
\begin{align*}
\mathcal{H}^{p, q}_{b,m}(X)\simeq H^{p,q}_{b,m}(X).
\end{align*}
In particular, we have
\begin{align*}
\dim H^{p,q}_{b,m}(X)<\infty, \forall m\in \mathbb{Z}, \forall~ 0\leq p, q\leq n-1.
\end{align*}

\end{thm}

For $\lambda\geq 0$, we collect the eigenspace of $\Box^{(p,q)}_{b,m}$ whose eigenvalue is less than or equal to $\lambda$ and define
\begin{align*}
\mathcal{H}^{p,q}_{b,m,\leq \lambda}&:=\oplus_{\sigma\leq \lambda}\mathcal{H}^{p,q}_{b,m,\sigma}(X),\\
\mathscr{Z}^{p,q}_{b,m,\leq \lambda}&:=\mbox{Ker}\overline{\partial}_b\cap\mathcal{H}^{p,q}_{b,m,\leq\lambda}.
\end{align*}

The Szeg\"{o} kernel function of the space $\mathscr{Z}^{p,q}_{b,m,\leq \lambda}$ is defined as
\begin{align*}
\Pi^{p,q}_{m,\leq \lambda}(x):=\sum_{j=1}^{d_m}|g_j(x)|^2,
\end{align*}
where $\{g_j\}_{j=1}^{d_m}$ is any orthonormal basis for the space $\mathscr{Z}^{p,q}_{b,m,\leq \lambda}$.

It is easy to see that
\begin{align}\label{dimension}
\dim \mathscr{Z}^{p,q}_{b,m,\leq \lambda}=\int_X\Pi^{p,q}_{m,\leq \lambda}dv_X.
\end{align}

The extremal function $S^{p,q}_{m,\leq \lambda}$ for $y\in X$ is defined by
\begin{align*}
S^{p,q}_{m,\leq \lambda}(y):=\sup_{u\in \mathscr{Z}^{p,q}_{b,m,\leq \lambda},\|u\|=1}|u(y)|^2.
\end{align*}

The next lemma is classical in Bergman's theory of reproducing kernels.
\begin{lemma}[c.f. \cite{Ber02}]\label{Szego kernel}
	
	For any $y\in X$,
	\begin{align*}
	S^{p,q}_{m,\leq \sigma}(y)\leq \Pi^{p,q}_{m,\leq \sigma}(y)\leq \binom{n-1}{p}\binom{n-1}{q} S^{p,q}_{m,\leq \sigma}(y).
	\end{align*}

	In particular,
	\begin{align*}
	\int_X S^{p,q}_{m,\leq \sigma}(y)dv_X\leq \dim \mathscr{Z}^{p,q}_{b,m,\leq \lambda}\leq \binom{n-1}{p}\binom{n-1}{q}\int_X S^{p,q}_{m,\leq \sigma}(y)dv_X.
	\end{align*}
\end{lemma}
For the proof of the above Lemma, we refer to \cite[Page 308, Lemma 4.1]{Ber02}.

\begin{rmk}\label{Grauert tube}
A typical example of compact CR manifold with a transversal CR $S^1$-action is the Grauert tube.
Let $M$ be a compact Hermitian manifold of complex dimension $n$, $(L,h)\rightarrow M$ be a holomorphic
line bundle. Denote by $\Theta$  the curvature of $(L,h)$. Let $X$ be the circle bundle
$\{v\in L^*:|v|^2_{h^{-1}}=1\}$ over $M$. $X$ is a real hypersurface in the complex manifold $L^*$ which is
the boundary of the disc bundle $D=\{v\in L^*:|v|^2_{h^{-1}}<1 \}$, with the defining function
$\rho=|v|^2_{h^{-1}}-1$. The Levi form of $\rho$ restricted to the complex tangent plane of $X$ coincides
with the pull-back of $\Theta $ through the canonical projection $\pi:X\rightarrow M$. It is a well-known
fact to the expert (c.f. \cite[Theorem 1.2]{CHT15}) that
\begin{itemize}
\item the space  $\Omega^{p,q}_m(X)$ can be identified with the space $\Omega^{p,q}(M,L^m)$,
\item for each  integer $m$, we get a subcomplex $(\Omega^{p,\bullet}_m(X),\overline{\partial}_b)$ which is
isomorphic to the Dolbeault complex $(\Omega^{p,\bullet}(M,L^m),\overline{\partial})$, thus we get that
the Kohn-Rossi cohomology group $H^{p,q}_{b,m}(X)$ is isomorphic to the Dolbeault cohomology group $H^{p,q}(M,L^m)$.
  \end{itemize}

Grauert tube was first introduced by Grauert \cite{Gr62,Gr94}.  Grauert established the identification of sections
of line bundle $L$ over $M$ and CR functions on $X$. This identification was used by Zelditch \cite{Ze98} to
study the asymptotic expansion of Bergman kernels. Further developments (e.g. the identification of
$\Omega^{0,q}_m(X)$ with $\Omega^{0,q}(M,L^m)$ for $q\geq 0$), we refer to \cite{CHT15, MM06}.

It is worth to point out that, from almost the same proof of Theorem 1.2 in \cite{CHT15}, one can get that
$\mathcal{H}^{p,q}_{b,m,\leq \lambda}(X)\simeq \mathcal{H}^{p,q}_{\leq \lambda}(M, L^m)$.

Meanwhile, there are also many examples of compact CR manifolds with transversal CR $S^1$-action which are not
tube type  \cite{HL16}. For example, let $X=\{(z_1,z_2)\in \mathbb{C}^2: |z_1|^2+|z_1+z_2^2|^2+|z_2|^2=1\}$ which is a
compact CR manifold with a transversal CR $S^1$-action defined by
\begin{align*}
X\times S^1&\rightarrow X\\
(z_1, z_2)&\mapsto (e^{i\theta}z_1, e^{2i\theta}z_2).
\end{align*}
The $S^1$-action defined above is locally free and  free on a dense, connected  open subset
$\{(z_1, z_2)\in X: z_1\neq 0\}$. Note that the CR $S^1$ action on the boundary of a Grauert tube is globally free.
\end{rmk}

\section{Local picture: canonical local coordinates }\label{sect: local picture}
In this section, we draw the local picture for compcat connected CR manifolds with  transversal CR $S^1$-action.
The following result is due to Baouendi-Rothschild-Treves \cite{BRT85}.
\begin{thm}[c.f. \cite{BRT85}]\label{local coordinate}Let $X$ be a compact CR manifold of dim$_X=2n-1$,
$n\geq 2$ with a transversal CR $S^1$-action. Let $\langle \cdot|\cdot\rangle$ be the given $T$-rigid
Hermitian metric on $X$. For any point $x_0\in X$, there exists local coordinates
$(x_1,\cdots,x_{2n-1})=(z,\theta)=(z_1,\cdots,z_{n-1},\theta)$, $z_j=x_{2j-1}+ix_{2j}$,
$j=1,\cdots, n-1$, $x_{2n-1}=\theta$, defined in some small neighborhood
$D=\{(z,\theta):|z|<\varepsilon, |\theta|<\delta\}$ of $x_0$ such that
\begin{align*}
&T=\frac{\partial}{\partial \theta}\\
&Z_j=\frac{\partial}{\partial z_j}+i\frac{\partial\varphi(z)}{\partial z_j}\frac{\partial}
{\partial \theta}, j=1,\cdots, n-1,
\end{align*}
where $\{Z_j(x)\}^{n-1}_{j=1}$ form a basis of $T_x^{1,0}X$ for each $x\in D$ and
$\varphi(z)\in \mathcal{C}^\infty(D,\mathbb{R})$ is independent of $\theta$. Moreover, on $D$ we can take
$(z,\theta)$ and $\varphi$ so that $(z(x_0),\theta(x_0))=(0,0)$ and
$\varphi(z)=\sum_{j=1}^{n-1}\lambda_j|z_j|^2+O(|z|^3)$, $\forall (z,\theta)\in D$, where
$\{\lambda_j\}^{n-1}_{j=1}$ are the eigenvalues of Levi-form of $X$ at $x_0$ with respect to the
given $T$-rigid Hermitian metric on $X$.
\end{thm}

\begin{rmk}
It was proved in \cite{HL16} that if $x_0\in X_{reg}$, $\delta$ can be taken to be $\pi$, and if $x_0$ is
not a regular point, say $x_0\in X_k$, $\delta$ can be taken to be any positive number smaller than $\frac{\pi}{k}$.
\end{rmk}

From Definition \ref{levi form}, by easy computation, one can get that
\begin{prop}
Let $X$ be a compact CR manifold of dimension $2n-1$, $n\geq 2$ with a transversal CR $S^1$-action.
Let $\langle\cdot|\cdot\rangle$ be the given $T$-rigid Hermitian metric on $X$. Let $D$ be a canonical
local patch with canonical coordinates $(z,\theta,\varphi)$ such that $(z,\theta,\varphi)$ is trivial
at $x_0$ as in Theorem \ref{local coordinate}. Suppose that $X$ is  weakly pseudoconvex, then
$i\partial\overline{\partial}\varphi\geq 0$ as a $(1,1)$-form on $\widetilde{D}$.
\end{prop}

\begin{lemma}[ \cite{HL16}]\label{local section} Fix $x_0\in X$ and let
$D=\widetilde{D}\times (-\delta,\delta)\subset \mathbb{C}^{n-1}\times\mathbb{R}$ be a canonical local patch
with canonical coordinates $(z,\theta,\varphi)$ such that $(z,\theta,\varphi)$ is trivial at $x_0$. The
$T$-rigid Hermitian metric on $D$ induces an Hermitian metric on $T^{*1,0}$ in a standard way. Up to a
coordinate transformation if necessary, we can find orthonormal frame $\{e^j\}_{j=1}^{n-1}$ of $T^{*1,0}$
with respect to the fixed $T$-rigid Hermitian metric such that on $D$, we have
$e^j(x)=e^j(z)=dz_j+O(|z|), \forall x=(z,\theta)\in D$, $j=1,\cdots, n-1$. Moreover, if we denote by $dv_X$
the volume form with respect to the $T$-rigid Hermitian metric on $\mathbb{C}TX$, then on $D$ we have
$dv_X=\lambda(z)dv(z)d\theta$ with $\lambda(z)\in \mathcal{C}^\infty(\widetilde{D},\mathbb{R})$ which does
not depend on $\theta$ and $dv(z)=2^{n-1}dx_1\cdots dx_{2n-2}$.
\end{lemma}
\begin{rmk}[c.f. \cite{HL16}]\label{induced metric}
For any $x_0\in X$, let $D=\widetilde{D}\times (-\delta,\delta)\subset \mathbb{C}^{n-1}\times\mathbb{R}$ be a
canonical local patch with canonical coordinates $(z,\theta,\varphi)$ such that $(z,\theta,\varphi)$ is
trivial at $x_0$. We identify $\widetilde{D}$ with an open subset of $\mathbb{C}^{n-1}$ with complex
coordinates $z=(z_1,\cdots,z_{n-1})$. Since $\{dz_j\}_{j=1}^{n-1}$ is a frame of $T^{*1,0}D$ over $D$, we will
treat them as the frame of $T^{*1,0}\widetilde{D}$ which is the bundle of $(1,0)$-forms over the domain
$\widetilde{D}$. Let $(g^{\bar{k}j}(z))$ be the Hermitian metric on $T^{*1,0}\widetilde{D}$ defined in the
proof of Lemma \ref{local section}.  It induces  Hermitian metrics on $T^{1,0}\widetilde{D}$  and
$T^{*,p,q}\widetilde{D}$ in a canonical way. We denote by the induced Hermitian metric on
$T^{1,0}\widetilde{D}$ by $\omega$.
\end{rmk}
\begin{rmk}
Under  the local canonical coordinate and the metric chosen in  Lemma \ref{local section}, one
can see that on $\widetilde{D}$, $\omega=\sum_{j=1}^{n-1} e^j\wedge \overline{e}^j$, and
$\omega(x_0)=\sum_{j=1}^{n-1}dz_j\wedge d\overline{z}_j$. Then  the volume form on
$\widetilde{D}$ is given by $\omega^{n-1}:=\frac{\omega^{n-1}}{(n-1)!}=\lambda(z)dv(z)$. Here,
$\lambda(z)\in\mathcal{C}^\infty(\widetilde{D},\mathbb{R})$ is the function defined in Lemma \ref{local section}.
\end{rmk}

\section{Local representations  of $\overline{\partial}_b$, $\overline{\partial}^*_b$ and $\Box^{(p,q)}_{b,m}$}\label{sect: local rep}
By the same calculations as in \cite{HL16}, we have the following local representation of the operators mentioned above.
\begin{lemma}\label{debarb-debar}
	For all $u\in \Omega^{p,q}_m(X)$, on $D$ we have
	\begin{align*}
	\overline{\partial}_b &u=e^{im\theta}e^{-m\varphi}\overline{\partial}(e^{m\varphi}e^{-im\theta}u),\overline{\partial}^*_bu=e^{im\theta}e^{-m\varphi}\overline{\partial}^{*,2m\varphi}(e^{m\varphi}e^{-im\theta}u)\\
	&\Box^{(p,q)}_{b,m}u=e^{im\theta}e^{-m\varphi}\Box^{(p,q)}_{2m\varphi}(e^{m\varphi}e^{-im\theta}u).
	\end{align*}
\end{lemma}

Based on Lemma \ref{debarb-debar}, we go a little bit further by direct computations to get the following

\begin{lemma}\label{eigenvalue comparison}
	Suppose that $u\in \Omega^{p,q}_m(X)$ satisfies $\Box^{(p,q)}_{b,m} u=\lambda u$. We define  $\widetilde{u}:=e^{m\varphi}e^{-im\theta}u$, then $\widetilde{u}\in \Omega^{p,q}_m(\widetilde{D})$ and  the following equality holds on $\widetilde{D}$:
	\begin{align*}
	\Box^{(p,q)}_{2m\varphi}\widetilde{u}=\lambda \widetilde{u}.
	\end{align*}
	Furthermore, for any $u\in \mathcal{H}^{p,q}_{b,m,\leq \sigma}(X)$, we get a form $\widetilde{u}\in \mathcal{H}^{p,q}_{2m\varphi,\leq \sigma}(\widetilde{D})$, where $\mathcal{H}^{p,q}_{b,m,\leq \sigma}(X)$ (resp. $\mathcal{H}^{p,q}_{2m\varphi,\leq \sigma}(\widetilde{D})$) is the linear span of the eigenforms of $\Box^{(p,q)}_{b,m}$ (resp. $\Box^{(p,q)}_{2m\varphi}$ ) with eigenvalue less than or equal to $\sigma$ on $X$ (resp. on $\widetilde{D}$).
	
\end{lemma}


Now we recall the so-called Siu's $\partial\overline{\partial}$-formula.  Let $(L,h)$ be a holomorphic Hermitian line bundle over a compact complex $n$-fold $(X,\omega)$ and $\alpha$ be a $L$-valued $(n,q)$-form. The Hodge-$*$ operator is defined by the formula
\begin{align}\label{Hodge star}
\alpha\wedge \overline{*\alpha}=|\alpha|^2\omega_n,
\end{align}
where $\omega_n=\omega^n/n!$. We define an $(n-q,n-q)$-form $T_\alpha$ associated to $\alpha$ in a local trivialization as
\begin{align}\label{Talpha}
T_\alpha=c_{n-q}\gamma\wedge\overline{\gamma}e^{-\psi},
\end{align}
where $\gamma=*\alpha$, $c_{n-q}=i^{(n-q)^2}$ and $\psi$ defines the metric of $L$.  Note that the form $T_\alpha$ is  well defined globally.

\begin{lemma}[c.f. \cite{Ber02}]\label{Siu formula 1}
	Let $\alpha$ be an $L$-valued $(n,q)$-form. If $\alpha$ is $\overline{\partial}$-closed
	\begin{align*}
	i\partial\overline{\partial}(T_\alpha\wedge \omega_{q-1})\geq \big(-2Re\langle \Box \alpha,\alpha\rangle+\langle\Theta_L\wedge\Lambda\alpha,\alpha\rangle-c|\alpha|^2\big)\omega_n
	,\end{align*}
	where $\Theta_L$ is the curvature of $(L,h)$ and   locally $\Theta_L$ can be written as $\Theta_L=i\partial\overline{\partial}\psi$ if  $\psi$ is  the local potential of $h$, i.e. $h=e^{-\psi}$. The constant $c$ is equal to zero if $\overline{\partial}\omega_{q-1}=\overline{\partial}\omega_q=0$.
\end{lemma}
\begin{rmk}\label{Siu formula 2}The expressions of the operator $\overline{\partial}^*$ maybe different in the compact case and in the noncompact case. But if we consider it as the formal adjoint of the operator $\overline{\partial}$, the expressions should always be the same. If we consider the formal adjoint, the formula in Lemma \ref{Siu formula 1} is pointwise in its nature.
\end{rmk}


\section{A preparation for localization procedure:  the scaling technique}\label{scaling}
In this section, we prepare necessary tools for the localization precedure needed in the proof of Theorem \ref{main theorem}.  This is a slight  modification of results in \cite{HL16}, say from  forms of  $(0,q)$ type to forms of $p,q$ type.

Fix $x_0\in X$, we take canonical local patch
$D=\widetilde{D}\times (-\delta,\delta)=\{(z,\theta):|z|<\varepsilon, |\theta|<\delta\}$ with canonical
coordinates $(z,\theta,\varphi)$ such that $(z,\theta,\varphi)$ is trivial at $x_0$. In this section, we
identify $\widetilde{D}$ with an open subset of $\mathbb{C}^{n-1}=\mathbb{R}^{2n-2}$ with complex
coordinates $z=(z_1,\cdots, z_{n-1})$.  Let $L_1\in T^{1,0}\widetilde{D}, \cdots, L_{n-1}\in T^{1,0}\widetilde{D} $
be the dual frame of $e^1,\cdots,e^{n-1}$ with respect to the $T$-rigid Hermitian metric $\langle\cdot|\cdot\rangle$
defined in Remark \ref{induced metric}. Let $\omega$ be  the induced Hermitian metric  on $T^{1,0}\widetilde{D}$.

Let $\Omega^{p,q}(\widetilde{D})$ be the space of smooth $(p,q)$-forms
on $\widetilde{D}$ and let $\Omega^{p,q}_0(\widetilde{D})$ be the subspace of $\Omega^{p,q}(\widetilde{D})$ whose elements have compact support in $\widetilde{D}$.
Let $\big(\cdot|\cdot\big)_{2\varphi}$ be the weighted inner product on the space $\Omega^{p,q}_0(\widetilde{D})$
defined as follows:
\begin{align*}
\big( f|g\big)=\int_{\widetilde{D}}\langle f|g\rangle e^{-2\varphi(z)}\lambda(z)dv(z)
\end{align*}
where $f,g\in \Omega^{p,q}_0(\widetilde{D})$ and $\lambda(z)$ is as in Remark \ref{induced metric}. We denote
by $L^2_{(p,q)}(\widetilde{D},2\varphi)$ the completion of $\Omega^{p,q}_0(\widetilde{D})$ with respect to
$\big(\cdot|\cdot\big)_{2\varphi}$. For $r>0$, let $\widetilde{D}_r=\{z\in \mathbb{C}^{n-1}:|z|<r\}$. Here
$\{z\in \mathbb{C}^{n-1}:|z|<r\}$ means that $\{z\in \mathbb{C}^{n-1}:|z_j|<r,j=1,\cdots,n-1\}$. For
$m\in \mathbb{N}$, let $F_m$ be the scaling map
$F_m(z)=\big(\frac{z_1}{\sqrt{m}},\cdots,\frac{z_{n-1}}{\sqrt{m}}\big)$, $z\in \widetilde{D}_{\log m}$.
From now on, we assume $m$ is sufficiently large such that $F_m(\widetilde{D}_{\log m})\subset\subset \widetilde{D}$.
We define the scaled bundle $F^*_mT^{*p,q}\widetilde{D}$ on $\widetilde{D}_{\log m}$  to be the bundle whose
fiber at $z\in \widetilde{D}_{\log m}$ is
\begin{align*}
F^*_mT^{*p,q}\widetilde{D}|_z=\Big\{\mathop{{\sum}'}_{|I|=p,|J|=q}a_{IJ}e^I\big(
\frac{z}{\sqrt{m}}\big)\wedge \overline{e^J}\big(
\frac{z}{\sqrt{m}}\big): \\a_{IJ}\in \mathbb{C}, I,J\mbox{~ strictly~increasing}\Big\}.
\end{align*}
We take the Hermitian metric $\langle\cdot|\cdot\rangle_{F^*_m}$ on $F_m^*T^{*p,q}\widetilde{D}$ so that at
each point $z\in \widetilde{D}_{\log m}$,
\begin{align*}
\Big\{ e^I\big(
\frac{z}{\sqrt{m}}\big)\wedge \overline{e^J}\big(
\frac{z}{\sqrt{m}}\big):|I|=p,|J|=q, I,J\mbox{~ strictly~increasing}\Big\}
\end{align*}
is an orthonormal frame for $F_m^*T^{*p,q}\widetilde{D}$ on $\widetilde{D}_{\log m}$.

Let $F^*_m\Omega^{p,q}_0(\widetilde{D}_r)$ denote the space of smooth sections of $F_m^*\Omega^{p,q}(\widetilde{D}_r)$
whose elements have compact support in $\widetilde{D}_r$. Given $f\in \Omega^{p,q}(\widetilde{D}_r)$. We write
$f=\sum'_{|J|=q}f_{IJ}e^I\wedge\overline{e^J}$. We define the scaled form
$F^*_mf\in F^*\Omega^{p,q}(\widetilde{D}_{\log m})$ by
\begin{align*}
F^*_mf=\mathop{{\sum}'}_{IJ}f_{IJ}(\frac{z}{\sqrt{m}})e^I\big(
\frac{z}{\sqrt{m}}\big)\wedge \overline{e^J}\big(
\frac{z}{\sqrt{m}}\big)p, z\in \widetilde{D}_{\log m}.
\end{align*}
For brevity, we denote $F^*_mf$ by $f(\frac{z}{\sqrt{m}})$. Let $P$ be a partial differential operator of
order one on $F_m{\widetilde{D}_{\log m}}$ with $\mathcal{C}^\infty$ coefficients. We write
$P=\sum_{j=1}^{2n-2}a_j(z)\frac{\partial}{\partial x_j}$. The scaled partial differential operator $P_{(m)}$
on $\widetilde{D}_{\log m}$ is given by $P_{(m)}=\sum_{j=1}^{2n-2}F^*_ma_j\frac{\partial}{\partial x_j}$.
Let  $f\in \mathcal{C}^\infty(F_m(\widetilde{D}_{\log m}))$. We can check that
\begin{align}\label{scaled P}
P_{(m)}(F^*_mf)=\frac{1}{\sqrt{m}}F^*_m(Pf).
\end{align}
Let $\overline{\partial}:\Omega^{p,q}(\widetilde{D})\rightarrow \Omega^{p,q+1}(\widetilde{D})$ be the
Cauchy-Riemann operator and we have
\begin{align*}
\overline{\partial}=\sum^{n-1}_{j=1}\overline{e^j}(z)\wedge \overline{L}_j+\sum_{j=1}^{n-1}(\overline{\partial}\overline{e^j})(z)\wedge(\overline{e^j}(z)\wedge)^*
\end{align*}
where $(\overline{e^j}(z)\wedge)^*:T^{*p,q}\widetilde{D}\rightarrow T^{*p,q-1}\widetilde{D}$ is the adjoint  of $\overline{e^j}(z)\wedge$ with respect to the Hermitian metric $\langle\cdot|\cdot\rangle$ on $T^{*p,q}\widetilde{D}$ given in Remark \ref{induced metric}, $j=1,\cdots, n-1$. That is
\begin{align*}
\langle e^j(z)\wedge u|v\rangle=\langle u|(e^j(z)\wedge)^*v\rangle
\end{align*}
for all $u\in T^{*p,q-1}\widetilde{D}$, $v\in T^{*p,q}\widetilde{D}$.
The scaled differential operator $\overline{\partial}_{(m)}:F^*_m\Omega^{p,q}(\widetilde{D}_{\log m})\rightarrow F^*_m\Omega^{p,q+1}(\widetilde{D}_{\log m})$ is given by
\begin{align}\label{scaled debar}
\overline{\partial}_{(m)}=\sum_{j=1}^{n-1}\overline{e^j}(\frac{z}{\sqrt{m}})\wedge\overline{L}_{j,(m)}+\sum_{j=1}^{n-1}\frac{1}{\sqrt{m}}(\overline{\partial}\overline{e^j})(\frac{z}{\sqrt{m}})\wedge (\overline{e^j}(\frac{z}{\sqrt{m}}))^*.
\end{align}
Similarly, $(\overline{e^j}(\frac{z}{\sqrt{m}})\wedge)^*: F^*_mT^{*p,q}\widetilde{D}\rightarrow F^*_mT^{*p,q-1}\widetilde{D}$ is the adjoint of $\overline{e^j}(\frac{z}{\sqrt{m}})\wedge$ with respect to $\langle\cdot|\cdot\rangle_{F^*_m}$, $j=1,\cdots,n-1$. From (\ref{scaled P}) and (\ref{scaled debar}), $\overline{\partial}_{(m)}$ satisfies that
\begin{align*}
\overline{\partial}_{(m)}F^*_mf=\frac{1}{\sqrt{m}}F^*_m(\overline{\partial} f), ~\forall f\in \Omega^{p,q}(F_m(\widetilde{D}_{\log m})).
\end{align*}

On the space $F^*_m\Omega^{p,q}_0(\widetilde{D}_{\log m})$, we define the  weighted inner product  $\big(\cdot|\cdot\big)_{2mF^*_m\varphi}$    as follows:
\begin{align*}
\big(f|g\big)_{2mF^*_m\varphi}=\int_{\widetilde{D}_{\log m}}\langle f|g\rangle_{F^*_m}e^{-2mF^*_m\varphi}\lambda(\frac{z}{\sqrt{m}})dv(z).
\end{align*}
Let $\overline{\partial}^*_{(m)}:F^*_m\Omega^{p,q+1}(\widetilde{D}_{\log m})\rightarrow F^*_m\Omega^{p,q}(\widetilde{D}_{\log m})$ be the formal adjoint of $\overline{\partial}_{(m)}$ with respect to the weighted inner product $\big(\cdot|\cdot)_{2mF^*_m\varphi}$. Let $\overline{\partial}^{*,2m\varphi}:\Omega^{p,q+1}{\widetilde{D}}\rightarrow \Omega^{p,q}(\widetilde{D})$ be the formal adjoint of $\overline{\partial}$ with respect to the weighted inner product $\big(\cdot|\cdot\big)_{2m\varphi}$. Then we also have
\begin{align*}
\overline{\partial}^*_{(m)}F^*_mf=\frac{1}{\sqrt{m}}F^*_m(\overline{\partial}^{*,2m\varphi}f), ~\forall f\in \Omega^{p,q}(F_m(\widetilde{D}_{\log m})).
\end{align*}
We now define the scaled complex Laplacian $\Box^{(p,q)}_{(m)}:F^*_m\Omega^{p,q}(\widetilde{D}_{\log m})\rightarrow F^*_m\Omega^{p,q}(\widetilde{D}_{\log m})$ which is given by $\Box^{(p,q)}_{(m)}=\overline{\partial}^*_{(m)}\overline{\partial}_{(m)}+\overline{\partial}_{(m)}\overline{\partial}^*_{(m)}$. Then we can see that
\begin{align}\label{scaled Laplace}
\Box^{(p,q)}_{(m)}F^*_mf=\frac{1}{m}F^*_m(\Box^{(p,q)}_{2m\varphi}f),~\forall f\in \Omega^{p,q}(F_m(\widetilde{D}_{\log m})).
\end{align}
p
Here
\begin{align*}
\Box^{(p,q)}_{2m\varphi}=\overline{\partial}\overline{\partial}^{*,2m\varphi}+\overline{\partial}^{*,2m\varphi}\overline{\partial}:\Omega^{p,q}(\widetilde{D})\rightarrow \Omega^{p,q}(\widetilde{D})
\end{align*}
is the complex Laplacian with respect to the given Hermitian metric on $T^{*p,q}(\widetilde{D})$ and weight function $2m\varphi(z)$ on $\widetilde{D}$.

Since $2mF^*_m\varphi=2\Phi_0(z)+\frac{1}{\sqrt{m}}O(|z|^3)$, $\forall z\in \widetilde{D}_{\log m}$, where $\Phi_0(z)=\sum_{j=1}^{n-1}\lambda_j|z_j|^2$, we have
\begin{align*}
\lim_{m\rightarrow \infty}\sup_{\widetilde{D}_{\log m}}|\partial^\alpha_z(2mF^*_m\varphi-2\Phi_0)|=0,~\forall \alpha\in \mathbb{N}_0^{2n-2}.
\end{align*}

Consider $\mathbb{C}^{n-1}$. Let $\langle\cdot|\cdot\rangle_{\mathbb{C}^{n-1}}$ be the Hermitian metric with constant coefficients on $T^{*p,q}\mathbb{C}^{n-1}$,   such that at the origin, it is equal to $\omega(0)$. Let $\big(\cdot|\cdot\big)_{2\Phi_0}$ be the $L^2$ inner product on $\Omega_0^{p,q}(\mathbb{C}^{n-1})$ given by
\begin{align*}
\big( f|g\big)_{2\Phi_0}=\int_{\mathbb{C}^{n-1}}\langle f|g\rangle e^{-2\Phi_0(z)}\lambda(0)dv(z), f,g\in \Omega^{p,q}_0(\mathbb{C}^{n-1}),
\end{align*}
where $\lambda(0)$ is the value of the function $\lambda(z)$ given in Remark \ref{induced metric} at $x_0$.

Put
\begin{align}\label{model laplace}
\Box^{(p,q)}_{2\Phi_0}=\overline{\partial}\overline{\partial}^{*,2\Phi_0}+\overline{\partial}^{*,2\Phi_0}\overline{\partial}:\Omega^{p,q}(\mathbb{C}^{n-1})\rightarrow \Omega^{p,q}(\mathbb{C}^{n-1}),
\end{align}
where $\overline{\partial}^{*,2\Phi_0}$ is the formal adjoint of $\overline{\partial}$ with respect to $\big(\cdot|\cdot\big)_{2\Phi_0}$.

It is not difficult to check that
\begin{align}\label{asymptotic laplace}
\Box^{(p,q)}_{(m)}=\Box^{p,q}_{2\Phi_0}+\varepsilon_m\mathcal{P}_m
\end{align}
on $\widetilde{D}_{\log m}$, where $\mathcal{P}$ is a second order partial differential operator and all the coefficients of $\mathcal{P}_m$ are uniformly bounded with respect to $m$ in $\mathcal{C}^\mu(\widetilde{D}_{\log m})$ norm for every $\mu\in \mathbb{N}_0$ and $\varepsilon_m$ is a sequence ¡¡tending to zero as $m\rightarrow \infty$.

From G{\aa}rding's inequality together with Sobolev estimates for elliptic operator $\Box^{(p,q)}_{(m)}$, one can get the following
\begin{prop}[c.f. \cite{Ber02}]\label{pointwise}
Let  $u\in F^*_m\Omega^{p,q}(\widetilde{D}_{\log m})$. For every $r>0$ with $\widetilde{D}_r\subset\subset \widetilde{D}_{\log m}$, and every $k\in \mathbb{N}^+$ and $k>\frac{n-1}{2}$, there is a constant $C_{r,k}$ independent of $m$ such that
\begin{align*}
|u(0)|^2\leq C_{r,k}\Big(\|u\|^2_{2mF^*_m\varphi,\widetilde{D}_r}+\|(\Box^{(p,q)}_{(m)})^ku\|_{2mF^*_m\varphi,\widetilde{D}_r}\Big).
\end{align*}
\end{prop}




\section{Proof of the Theorem \ref{main theorem}}\label{proof}
Now, after a long way of preparing, we are on the way to prove Theorem \ref{main theorem}, in the following  three steps.


\textbf{Step 1.} Fix a point $x_0\in X$. From Lemma \ref{local coordinate} and Remark \ref{induced metric},  up to a coordinate transformation, we can choose a canonical local patch  $D=\widetilde{D}\times (-\delta,\delta)=\{(z,\theta):|z|<\varepsilon, |\theta|<\delta\}$ with canonical coordinates $(z,\theta,\varphi)$ such that $(z,\theta,\varphi)$ is trivial at $x_0$ and the metric  $\omega$  induced by the $T$-rigid Hermitian metric on $X$ be the Hermitian metric satisfies $\omega=\frac{i}{2}\partial\overline{\partial}|z|^2=:\beta$ at $x_0$. Let $u\in \mathcal{H}^{n-1,q}_{b,m,\leq \lambda}(X)$ such that $\|u\|=1$ and $\overline{\partial}_b u=0$. Set $\widetilde{u}=e^{m\varphi}e^{-im\theta}u$ on $\widetilde{D}$, then from Lemma \ref{debarb-debar} and Lemma \ref{eigenvalue comparison}, we know that $\widetilde{u}\in \mathcal{H}^{n-1,q}_{2m\varphi,\leq \lambda}(\widetilde{D})$ and $\overline{\partial}\widetilde{u}=0$. By the definition and Lemma \ref{debarb-debar}, it is easy to show that
\begin{align}
|u|^2&=|\widetilde{u}|^2e^{-2m\varphi},\label{norm relation 1}\\
 |\Box^{(n-1,q)}_{b,m}u|^2&=|\Box^{(n-1,q)}_{2m\varphi}\widetilde{u}|^2e^{-2m\varphi}.\label{norm relation 2}\end{align}

We construct a trivial holomorphic Hermitian line bundle $(L:=\widetilde{D}\times \mathbb{C}, h:=e^{-2m\varphi})$ over $\widetilde{D}$. From (\ref{norm relation 1}) and (\ref{norm relation 2}), one can identify $\widetilde{u}$ with an $L$-valued $(n-1,q)$ form on $\widetilde{D}$, i.e. a section of the bundle $\Omega^{n-1,q}\otimes L$ over $\widetilde{D}$, and $\Box^{(n-1,q)}_{2m\varphi}$ with the formal $\overline{\partial}$-Laplacian operator on $\widetilde{D}$ with respect to the induced  Hermitian metric $\omega$ (see Remark \ref{induced metric}) and the Hermitian metric $h$ of $L$ on $\widetilde{D}$. For this consideration,  we make the following notations  throughout this section
\begin{align}
\big[\widetilde{u}\big]^2&:=|\widetilde{u}|^2e^{-2m\varphi}\label{new notation 1}\\
\big[\Box^{(n-1,q)}_{2m\varphi}\widetilde{u}\big]^2&:=|\Box^{(n-1,q)}_{2m\varphi}\widetilde{u}|^2e^{-2m\varphi}.\label{new notation 2}
\end{align}

Now it is the right time for us to  introduce the strategy of Berndtsson \cite{Ber02}.  Since $X$ is  pseudoconvex, then from Proposition
\ref{levi form} we have that   $\Theta_L=i\partial\overline{\partial}\varphi\geq 0$

From Lemma \ref{Siu formula 1}, we get that
\begin{align}\label{ME 1}
i\partial\overline{\partial}(T_{\widetilde{u}}\wedge \omega_{q-1})\geq (-2Re\langle \Box_{2m\varphi}^{n-1,q}\widetilde{u}, \widetilde{u}\rangle-c\big[\widetilde{u}\big]^2\rangle)\omega_{n-1}.
\end{align}

For $r>0$ small, we define
\begin{align*}
\sigma(r)&:=\int_{|z|<r}\big[\widetilde{u}\big]^2\omega_{n-1}=\int_{|z|<r}T_{\widetilde{u}}\wedge\omega_q=:s^2(r)
,\\
\lambda(r)&:=\Big(\int_{|z|<r}\big[\Box^{n-1,q}_{2m\varphi}\widetilde{u}\big]^2\Big)^{1/2}.
\end{align*}

From Cauchy's inequality, we get that
\begin{align*}
\int_{|z|<r}\big[\Box^{n-1,q}_{2m\varphi}\widetilde{u}\big]\big[\widetilde{u}\big]\leq \lambda(r)\sigma(r)^{1/2}.
\end{align*}

Without loss of generality, we may assume that $\lambda\geq 1$.

From (\ref{ME 1}) we see that
\begin{align}\label{ME 2}
&\int_{|z|<r}(r^2-|z|^2)i\partial\overline{\partial}(T_{\widetilde{u}}\wedge\omega_{q-1})\\
&\geq -cr^2\sigma(r)-2r^2\int_{|z|<r}\big[\Box^{n-1,q}_{2m\varphi}\widetilde{u}\big]\big[\widetilde{u}\big]\omega_{n-1}.\notag
\end{align}

Applying Stokes' formula to the left hand side of (\ref{ME 2}), we get that
\begin{align}\label{ME 3}
2&\int_{|z|<r}iT_{\widetilde{u}}\wedge \omega_{q-1}\wedge \beta\\
&\leq \int_{|z|=r} iT_{\widetilde{u}}\wedge \omega_{q-1}\wedge \partial|z|^2+cr^2\sigma(r)+2r^2\sigma(r)^{1/2}\lambda(r).\notag
\end{align}

Since   $\omega$ is smooth and $\omega(0)=\beta$, up to shrinking the local patch if necessary,  we have that
\begin{align}\label{beta and omega}
(1-O(r))\omega\leq \beta\leq (1-O(r))\omega.
\end{align}

Note that  if $\omega=\beta$, the boundary term in (\ref{ME 3}) can be estimated by an integral with respect to surface measure
\begin{align*}
r\int_{|z|=r}[\widetilde{u}]^2dS,
\end{align*}
and this implies that in our case
\begin{align}\label{boundary estimate}
\int_{|z|=r}iT_{\widetilde{u}}\wedge\omega_{q-1}\wedge\partial|z|^2\leq r(1-O(r))\int_{|z|=r}\big[\widetilde{u}\big]^2(\omega_{n-1}/\beta_{n-1})dS.
\end{align}

However,
\begin{align}\label{differential}
\int_{|z|=r}\big[\widetilde{u}\big]^2(\omega_{n-1}/\beta_{n-1})dS=\sigma'(r).
\end{align}

From (\ref{ME 3}), (\ref{beta and omega}) and (\ref{differential}),  we get that
\begin{align}\label{quasi inequality}
2q(1-O(r))\sigma(r)\leq r\sigma'(r)+2r^2\sigma(r)^{1/2}\lambda(r),
\end{align}
by incorporating  the term $cr^2\sigma(r)$ in $O(r)\sigma(r)$.

Dividing  by $2rs(r)$ to both sides of (\ref{quasi inequality}), we  obtain
\begin{align}\label{ME 4}
q(1/r-O(1))s(r)\leq s'(r)+r\lambda(r).
\end{align}

We are going to prove
\begin{align*}
s(r)\leq C r^k\lambda^{k/2}
\end{align*}
for $k\leq q$ by induction over $k$.

The statement is trivial for $k=0$. In fact, from (\ref{norm relation 1}) and (\ref{new notation 1}), we have that
\begin{align*}
\sigma(r)=\int_{|z|<r}\big[\widetilde{u}\big]^2\omega_{n-1}=\frac{1}{2\delta}\int_{|z|<r,-\delta\leq \theta\leq \delta}|u|^2dv_X\leq \frac{1}{2\delta},
\end{align*}
since we have assumed that $\|u\|=1$.

Now we assume that it has been proved for a certain value of $k<q$. Then (\ref{ME 4}) implies
\begin{align}\label{induction inequality}
(k+1)(1/r-O(1))s(r)\leq s'(r)+r\lambda(r).
\end{align}

Since $\widetilde{u}\in \mathcal{H}^{n-1,q}_{2m\varphi,\leq \lambda}(\widetilde{D})$, the form $\Box^{n-1,q}_{2m\varphi}\widetilde{u}$ also lies in  $\mathcal{H}^{n-1,q}_{2m\varphi,\leq \lambda}(\widetilde{D})$, then by the induction hypothesis we get that
\begin{align}\label{estimate lambda r}
\lambda(r)\leq Cr^k\lambda^{k/2+1}.
\end{align}

From (\ref{induction inequality}) and (\ref{estimate lambda r}), we obtain that
\begin{align}\label{induction inequality 2}
(k+1)(1/r-O(1))s(r)\leq s'(r)+Cr^{k+1}\lambda^{k/2+1}.
\end{align}

Set
\begin{align*}
\Phi(r)=(k+1)\int (1/r-O(1))dr\sim (k+1)\log r
\end{align*}
and multiply (\ref{induction inequality 2}) by the integrating factor $e^{-\Phi(r)}$. The result is that
\begin{align*}
(se^{-\Phi})'\geq -C\lambda^{k/2+1}.
\end{align*}

Integrate this inequality from $r$ to $\lambda^{-1/2}$. Since $e^{-\Phi}\sim 1/r^{k+1}$, we get that
\begin{align*}
r^{-(k+1)}s(r)\leq C \lambda^{k/2+1/2}+ s( \lambda^{-1/2})\lambda^{k/2+1/2}\leq C \lambda^{k/2+1/2}.
\end{align*}

By induction, we obtain that
\begin{align*}
s(r)\leq C r^q\lambda^{q/2}.
\end{align*}

After squaring both sides, we obtain that
\begin{align}\label{integral estimate}
\int_{|z|<r}\big[\widetilde{u}\big]^2\omega_{n-1}\leq C r^{2q}\lambda^q.
\end{align}

Go through the proof given above line by line, one can see that the constant $C$ only depends on the local coordinate, $c$ in Siu's formula (which depends only on the metric $\omega$), $O(1)$ and $\delta$, but from the compactness of $X$, one can get a uniform constant $C$ independent of $r$, $m$, $\lambda$ and the point $x_0$.


\textbf{Step 2.}   In the sequel, we shall use the scaling technique in Section \ref{scaling}.

For any form $u\in \Omega^{n-1,q}_{b,m}(X)$, we express $u$ in terms of the trivialization and local canonical coordinates on $D$ and write  $\widetilde{u}=e^{m\varphi}e^{-im\theta}u$ on $\widetilde{D}$ as before. Firstly we assume that $\lambda\leq m$. Put
\begin{align*}
\widetilde{u}^{(m)}(z)=F^*_m\widetilde{u}(z)=\widetilde{u}(\frac{z}{\sqrt{m}}),
\end{align*}
so that $\widetilde{u}^{(m)}$ is defined for $|z|<1$ if $m$ is large enough.

 We also have the scaled Laplacian $\Box^{(n-1,q)}_{(m)}$, and from (\ref{scaled Laplace}), it   satisfies
\begin{align*}
m\Box^{(n-1,q)}_{(m)}\widetilde{u}^{(m)}=F^*_m(\Box^{(n-1,q)}_{2m\varphi}\widetilde{u})=:(\Box^{(n-1,q)}_{2m\varphi}\widetilde{u})^{(m)}.
\end{align*}

From (\ref{asymptotic laplace}), $\Box^{(n-1,q)}_{(m)}$ converges to a $m$-independent elliptic operator as $m\rightarrow \infty$ on a neighborhood of  $|z|\leq1$.

Therefore, from Propositon \ref{pointwise}, we obtain that
\begin{align}\label{pointwise 1}
&|u(0)|^2=[\widetilde{u}](0)\\
&\leq C_{k}\Big(\int_{|z|<1}\big[\widetilde{u}^{(m)}\big]^2\omega^{(m)}_{n-1}+\int_{|z|<1}\big[\Box^{(n-1,q)}_{(m)})^k\widetilde{u}^{(m)}\big]^2\omega^{(m)}_{n-1}\Big),\notag
\end{align}
for $m$ sufficiently large and  $k>\frac{n-1}{2}$, where $C_{r,k}$ in Proposition \ref{pointwise} depends on $r$ and $k$, but here $C_{r,k}=C_{1,k}=:C_k$ only depends on $k$ since  $r=1$ in (\ref{pointwise 1}).

By coordinate transformation formula, we have that
\begin{align*}
\int_{|z|<1}\big[\widetilde{u}^{(m)}\big]^2\omega^{(m)}_{n-1}=m^{n-1}\int_{|z|<\frac{1}{\sqrt{m}}}\big[\widetilde{u}\big]^2\omega_{n-1},
\end{align*}
and
\begin{align*}
\int_{|z|<1}\big[\Box^{(n-1,q)}_{(m)})^k\widetilde{u}^{(m)}\big]^2\omega^{(m)}_{n-1}=m^{n-1-2k}\int_{|z|<\frac{1}{\sqrt{m}}}\big[(\Box^{n-1,q}_{2m\varphi})^k\widetilde{u}\big]^2\omega_{n-1}.
\end{align*}

From (\ref{integral estimate}) in Step 1,  we get that
\begin{align}\label{ME 5}
m^{n-1}\int_{|z|<\frac{1}{\sqrt{m}}}\big[\widetilde{u}\big]^2\omega_{n-1}\leq Cm^{n-1-q}(\lambda+1)^q,
\end{align}
and
\begin{align}\label{ME 6}
m^{n-1-2k}\int_{|z|<\frac{1}{\sqrt{m}}}\big[(\Box^{n-1,q}_{2m\varphi})^k\widetilde{u}\big]^2\omega_{n-1}\leq Cm^{n-1-q}(\lambda+1)^q(\lambda/m)^{2k}.
\end{align}

Combining (\ref{pointwise 1}), (\ref{ME 5}) and (\ref{ME 6}), we obtain that
\begin{align*}
|u(0)|^2\leq Cm^{n-1-q}(\lambda+1)^q.
\end{align*}

Secondly, if $\lambda\geq m$, we apply the above procedure to the scaling $\widetilde{u}^{(\lambda)}$ instead, and trivially get
\begin{align*}
|u(0)|^2\leq C\lambda^{n-1}.
\end{align*}


\textbf{Step 3.}  
 Since  $\overline{\partial}_b$ commutes with $\Box^{(p,q)}_{b,m}$, we have  the following exact sequence
\begin{align*}
0\rightarrow \mathscr{Z}^{n-1,q}_{b,m,\leq \lambda}\xrightarrow{inclusion} \mathcal{H}^{n-1,q}_{b,m,\leq\lambda}\xrightarrow{\overline{\partial}_b} \mathscr{Z}^{n-1,q+1}_{b,m,\leq \lambda}.
\end{align*}

Thus we obtain that
\begin{align}\label{dimension inequality}
\dim \mathcal{H}^{n-1,q}_{b,m,\leq\lambda}\leq \dim  \mathscr{Z}^{n-1,q}_{b,m,\leq \lambda}+\dim \mathscr{Z}^{n-1,q+1}_{b,m,\leq \lambda}.
\end{align}

From Lemma \ref{Szego kernel}, we see that, for any $y\in X$
\begin{align}\label{dimension integration 1}
\dim  \mathscr{Z}^{n-1,q}_{b,m,\leq \lambda}&\leq\binom{n-1}{p}\binom{n-1}{q}\int_X S^{n-1,q}_{m,\leq \lambda}(y)dv_X\\
&\leq Cm^{n-1-q}(\lambda+1)^q\notag
\end{align}
 with  $\lambda\leq m$, and
\begin{align}\label{dimension integration 2}
\dim  \mathscr{Z}^{n-1,q}_{b,m,\leq \lambda}\leq\binom{n-1}{p}\binom{n-1}{q}\int_X S^{n-1,q}_{m,\leq \lambda}(y)dv_X\leq C\lambda^{n-1}
\end{align}
with $\lambda\geq m$.

From  (\ref{dimension inequality}), (\ref{dimension integration 1}) and (\ref{dimension integration 2}), we obtain that for $\lambda\leq m$
\begin{align*}
\dim \mathcal{H}^{n-1,q}_{b,m,\leq\lambda}&\leq C\Big(m^{n-1-q}(\lambda+1)^q+m^{n-2-q}(\lambda+1)^{q+1}\Big)\\
&\leq Cm^{n-1-q}(\lambda+1)^q,
\end{align*}
and for $\lambda\geq m$,
\begin{align*}
\dim \mathcal{H}^{n-1,q}_{b,m,\leq\lambda}\leq C\lambda^n.
\end{align*}

In conclusion, we complete the proof of the  Theorem \ref{main theorem}.

\begin{rmk}
In  \cite{Ber02}, Berndtsson constructed  for any $0\leq q\leq n$,  a compact K\"{a}hler manifold $M$ of complex dimension $n$ and a semi-positive line bundle over $M$ such that for large $k$,
\begin{align*}
\dim \mathcal{H}^{n,q}_{\leq \lambda}(L^k)\geq C(\lambda+1)^qk^{n-q}.
\end{align*}
Let $X$ be the circle bundle $X:=\{v\in L^*:|v|^2_{h^{-1}}=1\}$ over $M$ which is of real dimension $2n+1$.
Then from the Theory of Grauert tube in Section \ref{Grauert tube}, one can see that  for large $m$,
\begin{align*}
\dim \mathcal{H}^{n,q}_{b,m,\leq \lambda} \geq C(\lambda+1)^q m^{n-q}.
\end{align*}
Thus we can see that in Theorem \ref{main theorem}, the growth order is sharp.
\end{rmk}

\section{Serre type duality theorem}\label{Serre}

Let $X$ be a compact $CR$-manifold of real dimension $2n-1$, which admits a transversal CR $S^1$-action. Let $u\in \mathcal{H}^{p,q}_m(X)$, which means that $\Box^{(p,q)}_{b,m}u=0$, which is equivalent to the fact that $\overline{\partial}_bu=0$ and $\overline{\partial}^*_bu=0$ from the Hodge theory of $\Box^{p,q}_{b,m}$. We  define the Hodge-$*$ operator in the CR level  by the following
\begin{align}\label{definition of *}
\langle u|v\rangle dv_X=u\wedge*\overline{v}\wedge\omega_0,
\end{align}
where $u,v\in \Omega^{p,q}_m(X)$, $dv_X$ is the volume form on $X$ defined in Lemma \ref{local section} , and $\omega_0$ is the global $1$-form associated to the action of $S^1$.

\begin{prop}\label{*formula}The Hodge-$*$ operator is a complex linear operator
\begin{align*}
*:\Omega^{p,q}_m(X)\rightarrow \Omega^{n-1-q,n-1-p}_m(X).
\end{align*}
\end{prop}
\begin{proof} For the proof, we follow the counterpart for complex manifold case in \cite{Kod}. Suppose that $u, v\in \Omega^{p,q}_m(X)$.
 For simplicity we denote as follows
\begin{align*}
A_p=\alpha_1& \cdots \alpha_p,~~B_q=\beta_1\cdots\beta_q~~~~\mbox{with}~~~~\alpha_1<\cdots<\alpha_p,~~\beta_1<\cdots<\beta_q
\\
&dz^{A_p}=dz^{\alpha_1}\wedge\cdots\wedge dz^{\alpha_p},~~dz^{B_q}=dz^{\beta_1}\wedge\cdots\wedge dz^{\beta_q}.
\end{align*}
Moreover, for $A_p=\alpha_1\cdots\alpha_p$, we put
\begin{align*}
A_{n-1-p}=\alpha_{p+1}\cdots\alpha_{n-1},
\end{align*}
where $\alpha_{p+1}<\cdots<\alpha_n$ and $\{\alpha_1,\cdots,\alpha_p,\alpha_{p+1},\cdots,\alpha_{n-1}\}$ is a permutation of $\{1,\cdots,n-1\}$. Similarly we define $B_{n-1-q}$ for a given $B_q=\beta_1\cdots\beta_q$.

Then with this notation, we can write $u, v$ in local coordinates by  \begin{align*}
&u=\mathop{{\sum}'}_{A_p,\overline{B}_q}u_{A_p\overline{B}_qJ}(z,\theta)dz^{A_p}\wedge d\overline{z^{B_q}}\\
&v=\mathop{{\sum}'}_{A_p,\overline{B}_q}v_{A_p\overline{B}_qJ}(z,\theta)dz^{A_p}\wedge d\overline{z^{B_q}},\end{align*}
where $dz^{A_p}$ and $d\overline{z^{B_q}}$ are $T$-invariant forms.

Set
\begin{align*}
\mbox{sgn}\left(
            \begin{array}{cc}
              A_p & A_{n-1-p} \\
              B_q & B_{n-1-q} \\
            \end{array}
          \right)=\mbox{sgn}\left(
                              \begin{array}{cccccc}
                                \alpha_1 & \cdots & \alpha_p & \alpha_{p+1} & \cdots & \alpha_{n-1} \\
                                \beta_1 & \cdots & \beta_q & \beta_{q+1} & \cdots & \beta_{n-1} \\
                              \end{array}
                            \right),
\end{align*}

From the definition of $(\cdot|\cdot)$ and similar computations in \cite{Kod}, one can compute that $*\overline{v}$ can be represented by
\begin{align}\label{*barv}
*\overline{v}=&(i)^{n-1}(-1)^k\sum_{A_p,B_q}\mbox{sgn}\left(
            \begin{array}{cc}
              A_p & A_{n-1-p} \\
              B_q & B_{n-1-q} \\
            \end{array}
          \right)\\
        &  \cdot g(z)\overline{v}^{\overline{B}_q A_p}(z,\theta)dz^{A_{n-1-p}}\wedge d\overline{z^{B_{n-1-q}}},\notag
\end{align}
 with $k={(n-1)(n-2)/2+(n-1)q}$, $g(z)=\det(g_{i,\overline{j}}(z))$  and
 \begin{align*}
 \overline{v}^{\overline{\beta}_1\cdots\overline{\beta}_q\alpha_1\cdots\alpha_p}(z,\theta)=\sum g^{\overline{\beta}_1\mu_1}\cdots g^{\overline{\beta}_q\mu_q}g^{\overline{\lambda}_1\alpha_1}\cdots g^{\overline{\lambda}_p\alpha_p}\overline{v}_{\mu_1\cdots\mu_q\overline{\lambda}_1\cdots\overline{\lambda_p}}(z,\theta).
 \end{align*}

 Replacing $\overline{v}$ by $v$ and interchanging $A$ with $B$, we obtain
 \begin{align}\label{*v}
 *v=&i^n(-1)^{k'}\sum_{A_p,B_q}\mbox{sgn}\left(
            \begin{array}{cc}
              A_p & A_{n-1-p} \\
              B_q & B_{n-1-q} \\
            \end{array}
          \right)\\
         & \cdot g(z)v^{\overline{A}_p B_q}(z,\theta)dz^{B_{n-1-q}}\wedge d\overline{z^{A_{n-1-p}}},\notag
 \end{align}
with $k'=(n-1)(n-2)/2+(n-1)p$, and

\begin{align*}
v^{\overline{A}_pB_q}(z,\theta)=\sum g^{\overline{\alpha}_1\lambda_1}\cdots g^{\overline{\alpha}_p\lambda_p}g^{\overline{\mu}_1\beta_1}\cdots g^{\overline{\mu}_q\beta_q}v_{\lambda_1\cdots\lambda_p\overline{\mu}_1\cdots\overline{\mu}_q}(z,\theta).
\end{align*}

It is easy to see that the map $v\rightarrow *v$ is linear, namely we have
\begin{align*}
*(c_1u+c_2v)=c_1*u+c_2*v
\end{align*}
for $c_1, c_2\in \mathbb{C}$.

From (\ref{*v}) , it is easy to see that 
\[T(*v)=im(*v), \mbox{ i.e. } *v\in \Omega^{n-1-q,n-1-p}_m(X).\]

\end{proof}

\begin{prop}\label{*bar}
We have $\overline{*v}=*\overline{v}$.
\end{prop}
\begin{proof}
It suffices to check this at arbitrarily  fixed point $x_0\in X$. Take canonical local coordinates $(z,\theta)$ so that $(z(x_0), \theta(x_0))=(0,0)$ and $g_{\alpha\overline{\beta}}(x_0)=\delta_{\alpha\beta}$. Then $g^{\overline{\alpha}\beta}(x_0)=\delta_{\alpha\beta}$, hence we have that $\overline{v}^{A_p\overline{B}_q}(x_0)=\overline{v_{A_p\overline{B}_q}(x_0)}$  and $v^{\overline{A}_pB_q}(x_0)=v_{A_p\overline{B}_q}(x_0)$. Moreover we have
\begin{align*}
\overline{dz^{B_{n-1-q}}\wedge \overline{dz^{A_{n-p}}}}=(-1)^{(n-1-q)(n-1-p)}dz^{A_{n-p}}\wedge {\overline{dz^{B_{n-1-q}}}}.
\end{align*}

Substituting these inequalities into (\ref{*barv}) and (\ref{*v}), we get that
\begin{align*}
\overline{*v(x_0)}=(-1)^{n-1+(n-1)p+(n-1-q)(n-1-p)+(n-1-p)q}*\overline{v}(x_0)=*\overline{v}(x_0).
\end{align*}
\end{proof}

By standard arguments, we have the following
\begin{prop}\label{double *}
If $v\in \Omega^{p,q}_m(X)$, then $**v=(-1)^{p+q}v$.
\end{prop}

In the sequel, we will deduce a formula for $\overline{\partial}_b^*$. 

Let $u\in \Omega^{p,q-1}_m(X)$ and $v\in \Omega^{p,q}_m(X)$. Recall that $\overline{\partial}_b^*$ is defined by the following formula
\begin{align}\label{definition}
(\overline{\partial}_bu|v)=(u|\overline{\partial}_b^*v).
\end{align}

By the definition of the $*$-operator, we have that
\begin{align*}
(\overline{\partial}_bu|v)=\int_X\overline{\partial}_bu\wedge*\overline{v}\wedge\omega_0.
\end{align*}

Since $X$ is compact, we have that
\begin{align}\label{stokes}
\int_Xd(u\wedge *\overline{v}\wedge\omega_0)=0.
\end{align}

Since $u\wedge*\overline{v}$ is an $(n-1,n-2)$ form, we have that 
\[\partial_b(u\wedge*\overline{v})=0 \mbox{ and }  \partial_b(u\wedge*\overline{v}\wedge\omega_0)=0. \]
Then we obtain that
\begin{align}\label{integration by part}
d(u\wedge*\overline{v}\wedge\omega_0)&=\partial_b(u\wedge*\overline{v}\wedge\omega_0)+\overline{\partial}_b(u\wedge*\overline{v}\wedge\omega_0)\\
&=\overline{\partial}_bu\wedge*\overline{v}\wedge\omega_0-(-1)^{p+q}u\wedge\overline{\partial}_b*\overline{v}\wedge\omega_0.
\end{align}

From (\ref{stokes}) and (\ref{integration by part}), we have that
\begin{align}\label{debarb star 1}
(\overline{\partial}_bu|v)=\int_X\overline{\partial}_bu\wedge*\overline{v}\wedge\omega_0=(-1)^{p+q}\int_Xu\wedge\overline{\partial}_b*\overline{v}\wedge\omega_0.
\end{align}

By Proposition \ref{double *}, we have that
\begin{align}\label{debarb star 2}
\int_Xu\wedge\overline{\partial}_b*\overline{v}\wedge\omega_0=-\int_Xu\wedge**\overline{\partial}_b*\overline{v}\wedge\omega_0.
\end{align}

Combining (\ref{definition}), (\ref{debarb star 1}) and (\ref{debarb star 2}), we have that
\begin{align}\label{dbarb star}
\overline{\partial}^*_b=-*\partial_b*.
\end{align}

Thus we have proved the following
\begin{prop}
$\overline{\partial}^*_b=-*\partial_b*$.
\end{prop}

Now let $u\in \mathcal{H}^{p,q}_{b,m}(X)$, i.e. $\overline{\partial}_bu=0$ and $\overline{\partial}^*_bu=0$. Set $v=*\overline{u}$. By Proposition \ref{*formula} and (\ref{*barv}), we have that $v\in \Omega^{n-1-p,n-1-q}_{-m}(X)$.
\begin{lemma}\label{dual form}
$\overline{\partial}_bv=0$ and $\overline{\partial}^*_bv=0$, which means that 
$$v\in \mathcal{H}^{n-1-p,n-1-q}_{b,-m}(X).$$
\end{lemma}
\begin{proof}
By direct computations, we have that
\begin{align*}
&\overline{\partial}_bv=\overline{\partial}_b*\overline{u}=(-1)^{p+q+1}\overline{**\partial_b*u}=(-1)^{p+q}\overline{*\overline{\partial}_b^*u}=0 .\\
&\overline{\partial}_b^*v=-*\partial_b**\overline{u}=(-1)^{p+q}*\overline{\overline{\partial}_bu}=0.
\end{align*}

Thus we complete the proof of the Lemma.
\end{proof}

Now we are on the way to get the following
\begin{thm}[= Theorem \ref{Serre duality I}]\label{Serre duality}Let $X$ be a compact $CR$-manifold of real dimension $2n-1$, which admits a transversal $S^1$-action. Then the Hodge $*$-operator defined by (\ref{definition of *}) induces a conjugate linear isomorphism
\begin{align*}
*:\mathcal{H}^{p,q}_{b,m}(X)&\rightarrow \mathcal{H}^{n-1-p,n-1-q}_{b,-m}(X)\\
u&\mapsto *\overline{u}.
\end{align*}
In particular, by combining the Hodge theory for $\Box^{p,q}_{b,m}$, we have the following conjugate line isomorphism in the cohomologcial level
\begin{align*}
H^{p,q}_{b,m}(X)\simeq H^{n-1-p,n-1-q}_{b,-m}(X).
\end{align*}
\end{thm}
\begin{proof}
Based on Lemma \ref{dual form}, we only need to prove that if $u\neq 0$, then $*\overline{u}\neq 0$, which follows directly from the definition of the $*$-operator.
\end{proof}
Combining Theorem \ref{main theorem}, we obtain the following
\begin{thm}[$=$ Theorem \ref{improved estimate}]\label{improved estimate 1}
Let $X$ be a compact $CR$-manifold of real dimension $2n-1$, which admits a transversal CR $S^1$-action. Suppose that $X$ is a weakly pseudo-convex CR manifold. Then we have that for $q\geq 1$,
\begin{align*}
\dim H^{0,q}_{b,-m}(X)\leq Cm^{q}, ~~~~\mbox{~~~as~~~}~~~~m\rightarrow +\infty.
\end{align*}
\end{thm}

\section{Applications}\label{application}

In this section, we introduce some applications of our main results.
\subsection{Morse     inequalities and Grauert-Riemenschneider    criterion}

Let $X$ be compact connected CR manifold of real dimension $2n-1$, $n\geq 2$ which admits a transversal CR $S^1$-action.  Set
$X(q):=\{x\in X|\mathcal{L}_x $ has exactly $q$ negative eigenvalues and $n-1-q$ positive eigenvalues$\}$.

In \cite{HL16},  it is proved that the following Morse type inequalities hold.

\begin{thm}[c.f. \cite{HL16}] Let $X$ be a compact connected CR manifold with a transversal CR $S^1$-action. Assume that dim$_{\mathbb{R}}X=2n-1$, $n\geq 2$. Then for every $q=0,1,\cdots,n-1$, as $m\rightarrow +\infty$, we have
\begin{align*}
\mbox{dim}H^{0,q}_{b,m}(X)&\leq \frac{m^{n-1}}{2\pi^n}\int_{X(q)}|\det \mathcal{L}_x|dv_{X}(x)+o(m^{n-1}), \\
\sum_{j=0}^q(-1)^{q-j} &\mbox{dim}H^{0,j}_{b,m}(X)\\
&\leq \frac{m^{n-1}}{2\pi^n}\sum^{q}_{j=0}(-1)^{q-j}\int_{X(j)}|\det\mathcal{L}_x|dv_X(x)+o(m^{n-1}).
\end{align*}
In particular, when $q=n-1$, as $m\rightarrow+ \infty$, we have the asymptotic Riemann-Roch theorem
\begin{align*}
\sum^{n-1}_{j=0}(-1)^j&\dim H^{0,j}_{b,m}(X)\\
&=\frac{m^{n-1}}{2\pi^n}\sum_{j=0}^{n-1}(-1)^j\int_{X(j)}|\det \mathcal{L}_x|dv_X(x)+o(m^{n-1}).
\end{align*}
\end{thm}

From Theorem \ref{Serre duality I}, we get that

\begin{thm}[=Theorem \ref{weak m negative-1}]\label{weak m negative}Let $X$ be a compact connected CR manifold with a transversal CR $S^1$-action. Assume that dim$_{\mathbb{R}}X=2n-1$, $n\geq 2$. Then for every $q=0,1,\cdots,n-1$, as $m\rightarrow +\infty$,  we have
\begin{align*}
&\mbox{dim}H^{n-1, q}_{b,-m}(X)\leq \frac{m^{n-1}}{2\pi^n}\int_{X(n-1-q)}|\det \mathcal{L}_x|dv_{X}(x)+o(m^{n-1}), \\
&\sum_{j=n-1-q}^{n-1}(-1)^{j+q-(n-1)} \mbox{dim}H^{n-1,j}_{b,-m}(X)\\
&\leq \frac{m^{n-1}}{2\pi^n}\sum^{n-1}_{j=n-1-q}(-1)^{q+j-(n-1)}\int_{X(n-1-j)}|\det\mathcal{L}_x|dv_X(x)+o(m^{n-1}).
\end{align*}
In particular, when $q=n-1$, as $m\rightarrow+ \infty$, we have the asymptotic Riemann-Roch theorem
\begin{align*}
&\sum^{n-1}_{j=0}(-1)^{n-1-j}\dim H^{n-1,j}_{b,m}(X)\\
&=\frac{m^{n-1}}{2\pi^n}\sum_{j=0}^{n-1}(-1)^{n-1-j}\int_{X(n-1-j)}|\det \mathcal{L}_x|dv_X(x)+o(m^{n-1}).
\end{align*}

\end{thm}

The following is the corresponding Morse type inequalities for $m\leq 0$ in \cite{HL16}.

\begin{thm} [c.f. \cite{HL16}]Let $X$ be a compact connected CR manifold with a transversal CR $S^1$-action. Assume that $\dim_{\mathbb{R}}X=2n-1$, $n\geq 2$. For every $q=0,1,2,\cdots, n-1$, as $m\rightarrow -\infty$, we have
\begin{align*}
&\dim H^{0,q}_{b,m}(X)\leq \frac{|m|^{n-1}}{2\pi^n}\int_{X(n-1-q)}|\det\mathcal{L}_x|dv_X(x)+o(|m|^{n-1}),\\
&\sum^q_{j=0}(-1)^{q-j}\dim H^{0,j}_{b,m}(X)\\
&\leq \frac{|m|^{n-1}}{2\pi^n}\sum^{q}_{j=0}(-1)^{q-j}\int_{X(n-1-j)}|\det\mathcal{L}_x|dv_X(x)+o(|m|^{n-1}).
\end{align*}
In particular, when $q=n-1$, as $m\rightarrow -\infty$, we have the following asymptotic Riemann-Roch theorem
\begin{align*}
&\sum^{n-1}_{j=0}(-1)^j\dim H^{0,j}_{b,m}(X)\\
&=\frac{|m|^{n-1}}{2\pi^n}\sum_{j=0}^{n-1}(-1)^j\int_{X(n-1-j)}|\det \mathcal{L}_x|dv_X(x)+o(|m|^{n-1}).
\end{align*}
\end{thm}

Applying our Theorem \ref{Serre duality I}, we have that

\begin{thm} [=Theorem \ref{morse m positive-1}]\label{morse m positive}Let $X$ be a compact connected CR manifold with a transversal CR $S^1$-action. Assume that $\dim_{\mathbb{R}}X=2n-1$, $n\geq 2$. For every $q=0,1,2,\cdots, n-1$, as $m\rightarrow +\infty$, we have
\begin{align*}
&\dim H^{n-1,q}_{b,m}(X)\leq \frac{m^{n-1}}{2\pi^n}\int_{X(q)}|\det\mathcal{L}_x|dv_X(x)+o(m^{n-1}),\\
&\sum^{n-1}_{j=n-1-q}(-1)^{q+j-(n-1)}\dim H^{n-1,j}_{b,m}(X)\\
&\leq \frac{m^{n-1}}{2\pi^n}\sum^{n-1}_{j=n-1-q}(-1)^{q+j-(n-1)}\int_{X(j)}|\det\mathcal{L}_x|dv_X(x)+o(m^{n-1}).
\end{align*}
In particular, when $q=n-1$, as $m\rightarrow +\infty$, we have the following asymptotic Riemann-Roch theorem
\begin{align*}
&\sum^{n-1}_{j=0}(-1)^{n-1-j}\dim H^{n-1,j}_{b,m}(X)\\
&=\frac{m^{n-1}}{2\pi^n}\sum_{j=0}^{n-1}(-1)^{n-1-j}\int_{X(j)}|\det \mathcal{L}_x|dv_X(x)+o(m^{n-1}).
\end{align*}
\end{thm}

From the asymptotic Riemann-Roch theorem in Theorem \ref{morse m positive} and Theorem \ref{main theorem}, we conclude that
\begin{thm}[=Theorem \ref{GR criterion-1}]\label{GR criterion}
Let $(X, T^{1,0}X)$ be a compact connected CR manifold of dimension $2n-1$, $n\geq 2$, where $T^{1,0}X$ is the given CR structure on $X$. Assume that $X$ admits a transversal CR $S^1$-action. If  $X$ is weakly pseudoconvex and strongly pseudoconvex at a point, then
\begin{align*}
\dim H^{n-1,0}_{b,m}(X)\approx m^{n-1}\mbox{~as~~} m\rightarrow +\infty.
\end{align*}
That is to say, there are a lot of CR sections of the canonical bundle $K_X$ of $X$.
\end{thm}

By $A\approx B$, we mean that there is a positive constant C such that $C^{-1}\leq A/B\leq  C$.


\subsection{Application to orbifold}

Let us first recall some basics of orbifold (c.f. \cite{CM13, MM07}).  Let $X$ be a Hausdorff space. An ($\mathcal{C}^\infty$) orbifold chart for an open set $U\subset M$ is a triple  $({\widetilde{U}, G},\varphi)$, where $\widetilde{U}$ is a domain in $\mathbb{R}^n$, $G$ is a finite group acting effectively as automorphisms of $\widetilde{U}$, and $\varphi_U:\widetilde{U}\rightarrow U$ is a continuous map such taht $\varphi\circ \sigma=\varphi$ for all $\sigma\in G$, inducing a homeomorphism from the quotient space $\widetilde{U}/G$ onto $U$. An injection between two charts $(\widetilde{U}, G,\varphi)$ and $(\widetilde{U}', G',\varphi')$ is a ($\mathcal{C}^\infty$) embedding $\lambda:\widetilde{U}\rightarrow \widetilde{U}'$ such that $\varphi'\circ \lambda=\varphi$.

An orbifold atlas on $X$ is a family $\mathcal{V}=\{(\widetilde{U}_i,G_i, \varphi_i)\}$ of orbifold charts such that $\mathcal{U}=\{U_i\}$ is a covering of $X$, where $U_i=\varphi_i(\widetilde{U}_i)$ and  given two charts $(\widetilde{U}_i, G_i,\varphi_i)$ and $(\widetilde{U}'_i, G'_i,\varphi'_i)$ and $x\in U_i\cap U_j$, there exist a chart $(\widetilde{U}_k, G_k,\varphi_k)$ with $x\in U_k$ and injections $\lambda_{ik}:(\widetilde{U}_k,G_k,\varphi_k)\rightarrow (\widetilde{U}_i,G_i,\varphi_i)$, $\lambda_{jk}:(\widetilde{U}_k,G_k,\varphi_k)\rightarrow (\widetilde{U}_j,G_j,\varphi_j)$.

An orbifold atlas $\mathcal{V}'$ is said to be a refinement of $\mathcal{V}$ if there exists an injection of every chart of $\mathcal{V}'$ into some chart of $\mathcal{V}$.  An orbifold $\mathcal{X}=(X,\mathcal{V})$ is a Hausdorff space $X$ with a (maximal) orbifold atlas $\mathcal{V}$.

We can assume an additional structure such as orientation, Riemannian metric, almost-complex structure or complex structure, CR structure on every $\widetilde{V}$ in the orbifold atlas $\mathcal{V}$. We understand the morphisms (and the groups) preserve the specified structure. Thus we can define oriented, Riemannian, almost-complex or complex and  CR orbifolds.

\begin{rmk}Let $X$ be a smooth manifold of real dimension $m$, and $G$ be a compact Lie group acting on $X$ locally free. Then the quotient space $X/G$ is an orbifold. Conversely, any orbifold $X$ can be realized  in this way, with $G=O(m)$, the orthogonal group of degree $m$ over $\mathbb{R}$.
\end{rmk}

An orbifold vector bundle $E$ over an orbifold $\mathcal{X}=(X,\mathcal{V})$ is defined as follows: $E$ is an orbifold  and for $U\in \mathcal{U}$, $(\widetilde{E}_U, G^E_U, \widetilde{\varphi}_U:\widetilde{E}_U\rightarrow \widetilde{U})$ is a $G^E_U$-equivariant vector bundle and $(\widetilde{E}_U, G^E_U, \widetilde{\varphi}_U)$ (resp. $(\widetilde{U}, G_U=G^E_U/K^E_U, \varphi_U)$, $K^E_U=\mbox{Ker}(G^E_U\rightarrow \mbox{Diffeo}(\widetilde{U}))$) is the orbifold structure on $E$ (resp. $X$). If $G^E_U$ acts effectively on $\widetilde{U}$ for $U\in \mathcal{U}$, we call $E$ a proper orbifold vector bundle.

\begin{rmk}
Let $E$ be an orbifold vector bundle on $(X,\mathcal{V})$. For $U\in \mathcal{U}$, let $\widetilde{E^{pr}_U}$ be the maximal  $K^E_U$-invariant  subbundle of $\widetilde{E}_U$ on $\widetilde{U}$. Then $(G_U, \widetilde{E^{pr}_U})$ defines a proper orbifold vector bundle on $(X,\mathcal{V})$, denoted by $E^{pr}$.
\end{rmk}

From the above Remark, without loss of generality,  we only consider proper orbifold vector bundle throughout this paper.

Let $E\rightarrow X$ be an orbifold vector bundle. A section $s:X\rightarrow E$
 is called $\mathcal{C}^k$ if for each $U\in \mathcal{U}$, $s|_U$ is covered by a $G^E_U$-invariant $\mathcal{C}^k$ section $\widetilde{s}_U:\widetilde{U}\rightarrow \widetilde{E}_U$.

If $X$ is an oriented orbifold, we define the integral $\int_X\Phi$ for a form $\Phi$ over $X$ as follows: if Supp$(\Phi)\subset U\in \mathcal{U}$, then
\begin{align*}
\int_X\Phi:=\frac{1}{|G_U|}\int_{\widetilde{U}}\widetilde{\Phi}_U.
\end{align*}

Now let $(X, T^{1,0}X)$ be a CR orbifold, i.e., there is an orbifold atlas $\mathcal{V}$ on $X$, such that for every  $(\widetilde{U},G,\varphi)\in \mathcal{V}$, there is a CR structure $(\widetilde{U}, T^{1,0}\widetilde{U}) $ so that the group $G$ acts on and  preserves the CR structure on $\widetilde{U}$.

In a similar way, one can define the $\overline{\partial}_b$-complex
on  $X$. Furthermore, Theorem \ref{regular 1}, Theorem \ref{regular
2}, Theorem \ref{hodge} also hold for compact CR orbifold with
transversal CR $S^1$-action. 
Since on every orbifold chart $(\widetilde{U},G,\varphi)$, the $S^1$-action generates a $G$-invariant vector field $T$ on $\widetilde{U}$ which preserves the CR structure on $\widetilde{U}$ and satisfies the transversal condition, then by almost the same arguments as in \cite{BRT85}, we can see  that Bauoendi-Roth-Treves' Theorem \ref{local coordinate} on the existence of canonical local coordinates also  holds on the CR orbifold setting. Then similarly, we have the following


\begin{thm}[=Theorem \ref{main theorem orbifold-1}]\label{main theorem orbifold} Let $(X, T^{1,0}X)$ be a compact connected CR orbifold of dimension $2n-1$, $n\geq 2$, where $T^{1,0}X$ is the given CR structure on $X$. Assume that $X$ admits a transversal CR $S^1$ action and  $X$ is weakly pseudoconvex. Then for $m$ sufficiently large, if $0\leq \lambda\leq m$,
\begin{align*}\dim \mathcal{H}^{n-1,q}_{b,m,\leq \lambda}\leq C(\lambda+1)^qm^{n-1-q},
\end{align*}
and if $1\leq m\leq \lambda$,
\begin{align*}
\dim \mathcal{H}^{n-1,q}_{b,m,\leq \lambda}\leq C\lambda^{n-1}.
\end{align*}

\end{thm}

Here we omit the details and  just give  a sketch of the proof. We take the  advantage that the  key estimate in our proof of Theorem \ref{main theorem} is a pointwise estimate. We work on  a local canonical orbifold chart $(\widetilde{U}, G, \varphi)$.  Note that  everything on $\widetilde{U}$ is invariant under the group action $G$, then  the desired estimate on the orbifold chart above can be achieved by the same method as before, which can be  naturally pushed-down to $U$ by the invariance under the group action. Then from the compactness of the orbifold, we can complete the proof of Theorem \ref{main theorem orbifold}.

\begin{rmk}
Similarly, Theorem \ref{Serre duality} and Theorem \ref{improved estimate} can also be generalized to CR orbifold setting.
\end{rmk}

Let $M$ be a compact connected complex manifold and $G$ be a compact Lie group acts analytically on $M$. We assume that the action of $G$ on $M$ is locally free. Then  $M/G$ is a complex orbifold. Suppose that $dim_{\mathbb{C}}(M/G)=n$. Let $L\rightarrow M$ be a $G$-invariant holomorphic line bundle over $M$, i.e., the transition functions of $L$ are $G$-invariant. Suppose that $L$ admits a locally free $G$-action compatible with that on $M$, i.e., an action $(g,v)(\in G\times L)\mapsto g\circ v\in L  $ with the property $\pi(g\circ v)=g\circ (\pi(v))$ where $\pi:L\rightarrow M$ is  the bundle projection. Then $L/G$ is an orbifold holomorphic line bundle over $M/G$.

The $G$-action on $L$ can be naturally extended to $L^m:=L^{\otimes m}$ and $L^*$ (the dual line bundle of $L$). Then $L^m/G$ and $L^*/G$ are also orbifold holomorphic line bundles over $M/G$. Put for $p,q=0,\cdots, n$,
 \[\Omega^{p,q}(M/G, L^m/G):=\{u\in \Omega^{p,q}(M,L^m): g^*u=u, \forall g\in G \}.\]

Since the Cauchy-Riemann operator  is $G$-invariant, we have  the    $\overline{\partial}$-complex
$$(\overline{\partial}, \Omega^{p,\bullet}(M/G, L^m/G)),~~~~p=0,1,\cdots,n,$$
and the $(p,q)$-th Dolbeault cohomology group:
\begin{align*}
H^{p,q}&(M/G, L^m/G)\\
&:=\frac{\mbox{Ker}\overline{\partial}:\Omega^{p,q}(M/G, L^m/G)\rightarrow \Omega^{p,q+1}(M/G, L^m/G)}{\mbox{Im}\overline{\partial}:\Omega^{p,q-1}(M/G, L^m/G)\rightarrow \Omega^{p,q+1}(M/G, L^m/G)}.
\end{align*}

Take any $G$-invariant Hermitian metric $h^L$   on $L$, it induces  a $G$-invariant  Hermitian metric $h^{L^*}$ on $L^*$, set $\widetilde{X}=\{v\in L^*| |v|^2_{h^{L^*}}=1\}$.  Then  $X=\widetilde{X}/G$ is a compact CR orbifold, and the natural $S^1$-action on $\widetilde{X}$ induces a locally free $S^1$-action on $X$ which can be verified that the action is CR and transversal.

\begin{thm}[= Theorem \ref{identification-1}] \label{identification}For every $p,q=0,1,\cdots,n$ and every $m\in \mathbb{Z}$, there is a bijective map $A^{(p,q)}_m:\Omega^{(p,q)}_m(X)\rightarrow \Omega^{(p,q)}(M, L^m)$ such that $A^{(p,q+1)}_m\overline{\partial}_{b,m}=\overline{\partial}A^{(p,q)}_m$ on $\Omega^{(p,q)}_m(X)$. Thus we have that
\begin{align*}
\Omega^{p,q}_m(X)&\simeq \Omega^{p,q}(M/G,L^m/G)\\
H^{p,q}_{b,m}(X)&\simeq H^{p,q}(M/G,L^m/G).
\end{align*}
In particular, $\dim H^{p,q}_{b,m}(X)<\infty$.
\end{thm}
\begin{proof}
	This is a small modification of the proof for the case $p=0$ in \cite{CHT15}.
 The local orbifold structure of $L/G$ is the following commutative diagram:
 \begin{align*}
 \xymatrix{
 \widetilde{U}^*\ar[d]_{\pi_{\widetilde{U}^*}}\ar[r]^{\varphi^*}
 &U^*\ar[d]^{\pi_U}\\
 \widetilde{U}\ar[r]_{\varphi}
 & U,}
 \end{align*}
 where $\widetilde{U}^*=\widetilde{U}\times \mathbb{C}$, $\pi_{\widetilde{U}^*}$ and $\pi_U$ are the projetions, $(\widetilde{U}^*,G_{\widetilde{U}^*},\varphi^*)$ and $(\widetilde{U},G_{\widetilde{U}},\varphi)$ are orbifold charts of $L/G$ and $M/G$ respectively. Let $s$ be a local trivializing section of $L/G$ defined on $U$, it corresponds to a section $\widetilde{s}$ of $\widetilde{U}^*$ on $\widetilde{U}$ with the property $(g\circ \widetilde{s})(x)=\widetilde{s}(g\circ x)$ for $g\in G$ and $x\in \widetilde{U}$.   Let $|\widetilde{s}|^2_{h^L}=e^{-2\widetilde{\psi}}$ on $\widetilde{U}$. Since $h^L$ is a $G$-invariant Hermitian metric on $L$, we have that $\widetilde{\psi}$ is a $G$-invariant function on $\widetilde{U}$ and   $G$ acts on $\widetilde{X}$ naturally.  The local orbifold structure of $X$ is the following commutative diagram:
 \begin{align*}
 \xymatrix{
\widetilde{U}\times S^1  \ar[d]_{\pi_{\widetilde{U}}}\ar[r]^{\varphi^*}
 & (\widetilde{U}\times S^1)/G_{\widetilde{U}^*}\ar[d]^{\pi_U}\\
 \widetilde{U}\ar[r]_{\varphi}
 & U,}
 \end{align*}
 where $(\widetilde{U}\times S^1)/G_{\widetilde{U}^*}\subset X$ and $U\subset M/G$.

  We identify $\widetilde{U}$ with an open set of $\mathbb{C}^n$, and introduce  holomorphic coordinates $z=(z_1,\cdots, z_n)$ on $\widetilde{U}$.  We have the local  diffeomorphism
  \begin{align*}
  \tau:\widetilde{U}\times (-\varepsilon_0, \varepsilon_0)&\rightarrow  X, ~~~~0<\varepsilon_0\leq \pi,\\
  (z,\theta)&\mapsto e^{-\widetilde{\psi}}\widetilde{s}^*(z)e^{-i\theta}.
  \end{align*}
We understand the image of the map $\tau$ is contained in  an orbifold chart of $X$. \\

  Put $D=\widetilde{U}\times (-\varepsilon_0,\varepsilon_0) $ as a canonical coordinate patch with $(z,\theta)$ canonical coordinates  (with respect to the trivialization $\widetilde{s}$ on $\widetilde{U}$) such that on $D$, the global real vector field $T$ induced by the $S^1$-action is $\frac{\partial}{\partial\theta}$  and
  \begin{align}\label{tangent}
  T^{1,0}\widetilde{U}&=\{\frac{\partial}{\partial z_j}-i\frac{\partial\widetilde{\psi}}{\partial z_j}(z)\frac{\partial}{\partial \theta}; j=1,2,\cdots, n\},\\
  T^{0,1}\widetilde{U}&=\{\frac{\partial}{\partial \overline{z}_j}+i\frac{\partial\widetilde{\psi}}{\partial \overline{z}_j}(z)\frac{\partial}{\partial \theta}; j=1,2,\cdots, n\},\notag
  \end{align}
and
\begin{align}\label{cotangent}T^{*1,0}\widetilde{U}=\{dz_j;j=1,2,\cdots,n\}, T^{*0,1}\widetilde{U}=\{d\overline{z}_j;j=1,2,\cdots,n\}.
\end{align}

We first define a local map $A^{p,q}_m(D):\Omega^{p,q}_m(D)\rightarrow \Omega^{p,q}(\widetilde{U},L^m)$ as follows
\begin{align*}
A^{p,q}_m(D):\Omega^{p,q}_m(D)&\rightarrow \Omega^{p,q}(\widetilde{U},L^m)\\
u(z,\theta)&\mapsto  \widetilde{s}^m(z)e^{m\widetilde{\psi}(z)}u(z,\theta)e^{-im\theta}.
\end{align*}
Note that $u(z,\theta)e^{-im\theta}$ is a form independent of $\theta$. It is easy to see that this map is a bijective.   We will prove that $A^{p,q}_m(D)$ can be patched to a global operator $A^{(p,q)}_m:\Omega^{(p,q)}_m(X)\rightarrow \Omega^{(p,q)}(M/G, L^m/G)$.

Let $s$ and $s_1$ be local trivializing sections of $L/G$ on an open set $U\subset M/G$, and $\widetilde{s}$ and $\widetilde{s}_1$ be the corresponding sections on $\widetilde{U}$.  Let $(z,\theta)\in \mathbb{C}^n\times \mathbb{R}$ and $(z,\eta)\in \mathbb{C}^n\times \mathbb{R}$ be canonical coordinates of $D$ with respect to $\widetilde{s}$ and $\widetilde{s}_1$ respectively.  Set $|\widetilde{s}|^2_{h^L}=e^{-2\widetilde{\psi}}$ and $|\widetilde{s}_1|^2_{h^L}=e^{-2\widetilde{\psi}_1}$. It suffices to prove that
\begin{align}\label{patch condition}
\widetilde{s}^m(z)e^{m\widetilde{\psi}(z)}u(z,\theta)e^{-im\theta}=\widetilde{s}_1^m(z)e^{m\widetilde{\psi}_1(z)}u_1(z,\eta)e^{-im\eta}
\end{align}

Observe that the above equality is invariant under the $G$-action.

Let $\widetilde{s}_1=g\widetilde{s}$ for $g$ a unit on $\widetilde{U}$. It is easy to find that
\begin{align}\label{weight transform}
\widetilde{\psi}_1=\widetilde{\psi}-\log|g|.
\end{align}

Note that if $\tau(z,\theta)=\tau_1(z,\eta)$, then with a certain branch of square root, we have
\begin{align}\label{theta-eta}
 e^{-i\theta}\big(\frac{g(z)}{\overline{g}(z)}\big)^{\frac{1}{2}}=e^{-i\eta}.
\end{align}

From (\ref{weight transform}) and (\ref{theta-eta}), we can easily verify that (\ref{patch condition}) holds.  Thus one proves that $A^{p,q}_m(D)$ can be patched to a global operator $A^{(p,q)}_m:\Omega^{(p,q)}_m(X)\rightarrow \Omega^{(p,q)}(M/G, L^m/G)$.

In the following, we  prove that  $A^{(p,q+1)}_m\overline{\partial}_{b,m}=\overline{\partial}A^{(p,q)}_m$ on $\Omega^{(p,q)}_m(X)$. Denote by $\widehat{u}(z)=u(z,\theta)e^{-im\theta}$,  which was previously known to be independent of $\theta$. From (\ref{tangent}) and (\ref{cotangent}),  by direct computations, one can see that
\begin{align}\label{db}
\overline{\partial}_b \widetilde{u}=\sum_{j=1}^n e^{im\theta} d\overline{z}_j\wedge(\frac{\partial\widehat{u}}{\partial \overline{z}_j}(z)+m\frac{\partial\widetilde{\psi}}{\partial \overline{z}_j}(z)\widehat{u}(z)).
\end{align}

Then  we have
\begin{align*}
A^{p,q+1}_m(D)(\overline{\partial}_bu)&=\widetilde{s}^m(z)e^{m\widetilde{\psi}(z)}(\sum_{j=1}^n d\overline{z}_j\wedge(\frac{\partial\widehat{u}}{\partial \overline{z}_j}(z)+m\frac{\partial\widetilde{\psi}}{\partial \overline{z}_j}(z)\widehat{u}(z))\\
&=\widetilde{s}^m(z)\overline{\partial}(e^{m\widetilde{\psi}}\widehat{u}(z))=\overline{\partial}A^{(p,q)}_m u.
\end{align*}
The last equality follows from the fact that $\widetilde{s}$ is holomorphic.

Since $u$ is  $G_{\widetilde{U}^*}$ invariant, $\widetilde{s}$ and $\widetilde{\psi}$ are $G_{\widetilde{U}}$ invariant, and the actions commutes with $\overline{\partial}$, we can see that $A^{p,q+1}_m(D)(\overline{\partial}_bu)$ is a well-defined local section of $L^m$ over $\widetilde{U}$, which can be   patched  together to a global section of $L^m/G$ over $M/G$ by the previous proof.

The proof of Theorem \ref{identification} is completed.
\end{proof}

Combining Theorem \ref{main theorem orbifold} and Theorem \ref{identification}, we can get the following
\begin{thm}[= Theorem \ref{orbifold cohomology-1}]\label{orbifold cohomology}
Let $M$ be a compact complex manifold and $G$ a compact Lie group. Suppose that $G$ acts on $M$ analytically, locally free and dim$_{\mathbb{C}}M/G=n$. Let $(L,h^L)$ be a $G$-invariant holomorphic Hermitian line bundle over $M$. Suppose that $L$ admits a locally free $G$-action compatible with $M$ and the curvature of $L$ is semi-positive. Then  we have that for $m$ sufficiently large,
\begin{align*}
\dim H^{n,q}(M/G, L^m/G)\leq Cm^{n-q},
\end{align*}
where $C$ is a constant independent of $m$.
\end{thm}

\begin{rmk}The above theorem corresponds to   Berndtsson's estimate in the orbifold case, which answers a folklore  open question, say generalizing Berndtsson's estimate to the orbifold setting,  informed to us by Hsiao.
	\end{rmk}

\subsection*{Acknowledgement}
The authors would like to  thank  Prof. Jih-Hsin Cheng and Prof. 
Chin-Yu Hsiao for their nice talks given at the Institute of
Mathematics in Chinese Academy of Sciences and for their valuable comments on
this paper.  The first author would like to  thank Dr. Xiaoshan Li
for helpful discussions on the Serre-type duality theorem.

\end{document}